\documentclass[reqno]{amsart}
\usepackage{amssymb,amsmath,amsthm,amsfonts,color}

\usepackage{mathrsfs,dsfont,comment}
\usepackage{hyperref}
\usepackage{enumerate,esint}
\usepackage{mathtools}
\DeclarePairedDelimiter{\ceil}{\lceil}{\rceil}
\DeclarePairedDelimiter{\floor}{\lfloor}{\rfloor}
\allowdisplaybreaks
\numberwithin{equation}{section}
\theoremstyle{plain}
\newtheorem{theorem}{Theorem}[section]

\newtheorem{lemma}[theorem]{Lemma}

\newtheorem{corollary}[theorem]{Corollary}
\theoremstyle{definition}
\newtheorem{definition}[theorem]{Definition}
\newtheorem{remark}[theorem]{Remark}

\newcommand\R{\mathbb R}
\newcommand\M{\mathbb M}
\newcommand\N{\mathbb N}

\newcommand\Z{\mathbb{Z}}

\newcommand\iO{\int_\Omega}
\newcommand{\iOQ}{\int_{\Omega}\int_Q}

\newcommand\iQ{\int_Q}
\newcommand\ep{\varepsilon}

\newcommand{\dist}{\mathrm{dist}}
\newcommand\wk{\rightharpoonup}
\newcommand\lrightharpoonup{-\hspace{-0.2cm}\rightharpoonup}

\newcommand\swkts{\overset{2-s}{\lrightharpoonup}}

\newcommand\oscz{\overset{osc, Z}{\lrightharpoonup}}
\newcommand\oscy{\overset{osc, Y}{\lrightharpoonup}}

\newcommand\iv{\big(-\tfrac12,\tfrac12\big)}

\newcommand\wks{\overset{dr-3-s}{\lrightharpoonup}}
\newcommand\stts{\overset{dr-3-s}{-\hspace{-0.2cm}\to}}
\newcommand\swks{\overset{3-s}{\lrightharpoonup}}
\newcommand\sstts{\overset{3-s}{-\hspace{-0.2cm}\to}}

\newcommand{\mthree}{\M^{3\times 3}}

\newcommand\nh{\nabla_h}
\newcommand\intoyz[1]{\iO\int_Q\int_Q#1\,dz\,dy\,dx}
\newcommand\intoy[1]{\iO\int_Q #1\,dy\,dx}
\newcommand\tv{\tilde{v}}
\newcommand\ttv{\tilde{\tilde{v}}}
\newcommand\intomyz[1]{\int_{\omega}\int_Q\int_Q #1\,dz\,dy\,dx'}
\newcommand\np{(\nabla')^{\perp}}
\newcommand\npp{\nabla^{\perp}}
\newcommand\ze{\Z^{\ep}}
\newcommand\zhe{\hat{\Z}^{\ep}}
\newcommand\zeb{\Z^{\ep}_b}
\newcommand\zeg{\Z^{\ep}_g}

\newcommand\szeb{\sum_{\lambda\in\zeb}}
\newcommand\sle{\sum_{\lambda\in (\zeb\cup\zeg)}}
\newcommand\slet{\sum_{\lambda\in \zhe}}
\newcommand\szeg{\sum_{\lambda\in\zeg}}

\newcommand\qel{Q(\ep^2(h)\lambda,\ep^2(h))}

\newcommand\be[1]{\begin{equation}\label{#1}}
\newcommand\ee{\end{equation}}
\newcommand\bm[1]{\begin{align}\label{#1}}
\newcommand\bmm{\begin{align*}}

\title [Multiscale homogenization in Kirchhoff's nonlinear plate theory]
{Multiscale homogenization in Kirchhoff's nonlinear plate theory}
\author[L. Bufford] {Laura Bufford} 
\address[Laura Bufford]{Department of Mathematics\\ Carnegie Mellon University\\Forbes Avenue\\Pittsburgh PA 15213}
\email[L. Bufford] {lbufford@andrew.cmu.edu} 
\author[E. Davoli] {Elisa Davoli} 
\address[Elisa Davoli]{Department of Mathematics\\ Carnegie Mellon University\\Forbes Avenue\\Pittsburgh PA 15213}
\email[E. Davoli]{edavoli@andrew.cmu.edu}
\author[I. Fonseca] {Irene Fonseca} 
\address[Irene Fonseca]{Department of Mathematics\\ Carnegie Mellon University\\Forbes Avenue\\Pittsburgh PA 15213}
\email[I. Fonseca]{fonseca@andrew.cmu.edu}

\begin{document} 
\vskip .2truecm
\begin{abstract}
\small{
The interplay between multiscale homogenization and dimension reduction for nonlinear elastic thin plates is analyzed in the case in which the scaling of the energy corresponds to Kirchhoff's nonlinear bending theory for plates. Different limit models are deduced depending on the relative ratio between the thickness parameter $h$ and the two homogenization scales $\ep$ and $\ep^2$.}
\end{abstract}
\maketitle
\section{Introduction}
The search for lower dimensional models describing thin three-dimensional structures is a classical problem in mechanics of materials. Since the early '90s it has been tackled successfully by means of variational tecniques, and starting from the seminal papers \cite{acerbi.buttazzo.percivale, friesecke.james.muller, friesecke.james.muller2, ledret.raoult} hierarchies of limit models have been deduced by $\Gamma$-convergence, depending on the scaling of the elastic energy with respect to the thickness parameter. 

The first homogenization results in nonlinear elasticity have been proved in \cite{braides} and \cite{muller}. In these two papers, A. Braides and S. M\"uller assume p-growth of a stored energy density $W$ that oscillates periodically in the in-plane direction. They show that as the periodicity scale goes to zero, the elastic energy $W$ converges to a homogenized energy, whose density is obtained by means of an infinite-cell homogenization formula. 

In \cite{babadjian.baia, braides.fonseca.francfort} the authors treat simultaneously homogenization and dimension reduction for thin plates, in the membrane regime and under p-growth assumptions of the stored energy density. More recently, in \cite{hornung.neukamm.velcic}, \cite{neukamm.velcic}, and \cite{velcic} models for homogenized plates have been derived under physical growth conditions for the energy density. We briefly describe these results. 

Let
$$\Omega_h:=\omega\times(-\tfrac h2,\tfrac h2)$$
be the reference configuration of a nonlinearly elastic thin plate, where $\omega$ is a bounded domain in $\R^2$, and $h>0$ is the thickness parameter. Assume that the physical structure of the plate is such that an in-plane homogeneity scale $\ep(h)$ arises, where $\{h\}$ and $\{\ep(h)\}$ are monotone decreasing sequences of positive numbers,  $h\to 0$, and $\ep(h)\to 0$ as $h\to 0$. In \cite{hornung.neukamm.velcic, neukamm.velcic, velcic} the rescaled nonlinear elastic energy associated to a deformation $v\in W^{1,2}(\Omega_h;\R^3)$ is given by
$$\mathcal{I}^h(v):=\frac{1}{h}\int_{\Omega_h}W\Big(\frac{x'}{\ep(h)},\nabla v(x)\Big)\,dx,$$
where $x':=(x_1,x_2)\in\omega$, and the stored energy density $W$ is periodic in its first argument and satisfies the commonly adopted assumptions in nonlinear elasticity, as well as a nondegeneracy condition in a neighborhood of the set of proper rotations.

In \cite{neukamm.velcic} the authors focus on the scaling of the energy corresponding to Von K\'arm\'an plate theory, that is they consider deformations $v^h\in W^{1,2}(\Omega_h;\R^3)$ such that
$$\limsup_{h\to 0}\frac{\mathcal{I}^h(v^h)}{h^4}<+\infty.$$
Under the assumption that the limit
$$\gamma_1:=\lim_{h\to 0}\frac{h}{\ep(h)}$$
exists, different homogenized limit models are identified, depending on the value of $\gamma_1\in[0,+\infty]$.

 A parallel analysis is carried in \cite{hornung.neukamm.velcic}, where the scaling of the energy associated to Kirchhoff's plate theory is studied, i.e., the deformations under consideration satisfy
 $$\limsup_{h\to 0}\frac{\mathcal{I}^h(v^h)}{h^2}<+\infty.$$
 In this situation a lack of compactness occurs when $\gamma_1=0$ (the periodicity scale tends to zero much more slowly than the thickness parameter). A partial solution to this problem, in the case in which
$$\gamma_2:=\lim_{h\to 0}\frac{h}{\ep^2(h)}=+\infty,$$ 
is proposed in \cite{velcic}, by means of a careful application of Friesecke, James and M\"uller's quantitative rigidity estimate, and a construction of piecewise constant rotations (see \cite[Theorem 4.1]{friesecke.james.muller} and \cite[Theorem 6]{friesecke.james.muller2} and \cite[Lemma 3.11]{velcic}). The analysis of simultaneous homogenization and dimension reduction for Kirchhoff's plate theory in the remaining regimes is still an open problem.\\

In this paper we deduce a multiscale version of the results in \cite{hornung.neukamm.velcic} and \cite{velcic}. We focus on the scaling of the energy which corresponds to Kirchhoff's plate theory, and we assume that the plate undergoes the action of two homogeneity scales - a coarser one and a finer one - i.e.,
the rescaled nonlinear elastic energy is given by
$$\mathcal{J}^h(v):=\frac{1}{h}\int_{\Omega_h}W\Big(\frac{x'}{\ep(h)},\frac{x'}{\ep^2(h)},\nabla v(x)\Big)\,dx$$
for every deformation $v\in W^{1,2}(\Omega_h;\R^3)$, where the stored energy density $W$ is periodic in its first two arguments and, again, satisfies the usual assumptions in nonlinear elasticity, as well as the nondegeneracy condition (see Section \ref{section:setting}) adopted in \cite{hornung.neukamm.velcic, neukamm.velcic, velcic}. We consider sequences of deformations $\{v^h\}\subset W^{1,2}(\Omega_{h};\R^3)$ verifying
\be{eq:intro-en-est}\limsup_{h\to 0}\frac{\mathcal{J}^h(v^h)}{h^2}<+\infty,\ee
and we seek to identify the effective energy associated to the rescaled elastic energies $\big\{\frac{\mathcal{J}^h(v^h)}{h^2}\big\}$ for different values of $\gamma_1$ and $\gamma_2$, i.e. depending on the interaction of the homogeneity scales with the thickness parameter.

As in \cite{hornung.neukamm.velcic}, a sequence of deformations satisfying \eqref{eq:intro-en-est} converges, up to the extraction of a subsequence, to a limit deformation $u\in W^{1,2}(\omega;\R^3)$ satisfying the isometric constraint
\be{eq:intro-normal1}\partial_{x_\alpha} u(x')\cdot\partial_{x_\beta} u(x')=\delta_{\alpha,\beta}\quad\text{for a.e. }x'\in\omega,\quad\,\alpha,\beta\in\{1,2\}. \ee
We will prove that the effective energy is given by
$$\mathcal{E}^{\gamma_1}(u):=\begin{cases}\frac{1}{12}\int_{\omega}\overline{\mathscr{Q}}_{\rm hom}^{\gamma_1}(\Pi^u(x'))\,dx'&\text{if }u\text{ satisfies }\eqref{eq:intro-normal1},\\
+\infty&\text{otherwise},\end{cases}$$
where $\Pi^u$ is the second fundamental form associated to $u$ (see \eqref{eq:def-Pi}), and $\overline{\mathscr{Q}}_{\rm hom}^{\gamma_1}$ is a quadratic form dependent on the value of $\gamma_1$, with explicit characterization provided in \eqref{eq:def-q-hom-bar-gamma1-infty}--\eqref{eq:def-q-hom-gamma1-infty}.
To be precise, our main result is the following.
\begin{theorem}
\label{thm:main-result}
Let $\gamma_1\in [0,+\infty]$ and let $\gamma_2=+\infty$. Let $\{v^h\}\subset W^{1,2}(\Omega_h;\R^3)$ be a sequence of deformations satisfying the uniform energy estimate \eqref{eq:intro-en-est}. There exists a map $u\in W^{2,2}(\omega;\R^3)$ verifying \eqref{eq:intro-normal1} such that, up to the extraction of a (not relabeled) subsequence, there holds
\begin{align}
\nonumber v^h(x',hx_3)-\fint_{\Omega_1} v^h(x',hx_3)\,dx\to u\quad\text{strongly in }L^2(\Omega_1;\R^3),\\
\nonumber \nh v^h(x',hx_3)\to (\nabla'u|n_u)\quad\text{strongly in }L^2(\Omega_1;\mthree),
\end{align}
with 
$$n_u(x'):=\partial_{x_1}u(x')\wedge\partial_{x_2}u(x')\quad\text{for a.e. }x'\in\omega,$$
and
\be{eq:intro-liminf}\liminf_{h\to 0}\frac{\mathcal{J}^h(v^h)}{h^2}\geq \mathcal{E}^{\gamma_1}(u).\ee
Moreover, for every $u\in W^{2,2}(\omega;\R^3)$ satisfying \eqref{eq:intro-normal1}, there exists a sequence $\{v^h\}\subset W^{1,2}(\Omega_h;\R^3)$ such that
\be{eq:intro-limsup}\limsup_{h\to 0}\frac{\mathcal{J}^h(v^h)}{h^2}\leq \mathcal{E}^{\gamma_1}(u).\ee
\end{theorem}
We remark that our main theorem is consistent with the results proved in \cite{hornung.neukamm.velcic} and \cite{velcic}. Indeed, in the presence of a single homogeneity scale, it follows directly from \eqref{eq:def-q-hom-bar-gamma1-infty}--\eqref{eq:def-q-hom-gamma1-infty} that $\overline{\mathscr{Q}}_{\rm hom}^{\gamma_1}$ reduces to the effective energy identified in \cite{hornung.neukamm.velcic} and \cite{velcic} for $\gamma_1\in (0,+\infty]$ and $\gamma_1=0$, respectively. The main difference with respect to \cite{hornung.neukamm.velcic} and \cite{velcic} is in the structure of the homogenized energy density $\overline{\mathscr{Q}}_{\rm hom}^{\gamma_1}$, which is obtained by means of a double pointwise minimization, first with respect to the faster periodicity scale, and then with respect to the slower one and the $x_3$ variable (see \eqref{eq:def-q-hom-bar-gamma1-infty}--\eqref{eq:def-q-hom-gamma1-infty}).\\

The quadratic behavior of the energy density around the set of proper rotations together with the linearization occurring due to the high scalings of the elastic energy yield a convex behavior for the homogenization problem, so that, despite the nonlinearity of the three-dimensional energies, the effective energy does not have an infinite-cell structure, in contrast with \cite{muller}. The main techniques for the proof of the liminf inequality \eqref{eq:intro-liminf} are
the notion of multiscale convergence introduced in \cite{allaire.briane}, and its adaptation to dimension reduction (see \cite{neukammPHD}). The proof of the limsup inequality \eqref{eq:intro-limsup} follows that of \cite[Theorem 2.4]{hornung.neukamm.velcic}.

The crucial part of the paper is the characterization of the three-scale limit of the sequence of linearized elastic stresses (see Section \ref{section:E}). We deal with sequences having unbounded $L^2$ norms but whose oscillations on the scale $\ep$ or $\ep^2$ are uniformly controlled. As in \cite[Lemmas 3.6--3.8]{hornung.neukamm.velcic}, to enhance their multiple-scales oscillatory behavior we work with suitable oscillatory test functions having vanishing average in their periodicity cell. 

The presence of three scales increases the technicality of the problem in all scaling regimes. For $\gamma_1\in(0,+\infty]$, Friesecke, James and M\"uller's rigidity estimate (\cite[Theorem 4.1]{friesecke.james.muller}) leads us to work with sequences of rotations that are piecewise constant on cubes of size $\ep(h)$ with centers in $\ep(h)\Z^2$. However, in order to identify the three-scale limit of the linearized stresses, we must consider sequences oscillating on a scale $\ep^2(h)$. This problem is solved in Step 1 of the proof of Theorem \ref{thm:limit stresses}, by subdividing the cubes of size $\ep^2(h)$, with centers in $\ep^2(h)\Z^2$, into ``good cubes" lying completely within a bigger cube of size $\ep(h)$ and center in $\ep(h)\Z^2$ and ``bad cubes", and by showing that the measure of  the intersection between $\omega$ and the set of ``bad cubes" converges to zero faster than or comparable to $\ep(h)$, as $h\to 0$.

The opposite problem arises in the case in which $\gamma_1=0$. By Friesecke, James and M\"uller's rigidity estimate (\cite[Theorem 4.1]{friesecke.james.muller}), it is natural to work with sequences of piecewise constant rotations which are constant on cubes of size $\ep^2(h)$ having centers in the grid $\ep^2(h)\Z^2$, whereas in order to identify the limit multiscale stress we need to deal with oscillating test functions with vanishing averages on a scale $\ep(h)$. The identification of ``good cubes" and ``bad cubes"  of size $\ep^2(h)$ is thus not helpful in this latter framework as the contribution of the oscillating test functions on cubes of size $\ep^2(h)$ is not negligible anymore. Therefore, we are only able to perform an identification of the multiscale limit in the case $\gamma_2=+\infty$, extending to the multiscale setting the results in \cite{velcic}. The identification of the effective energy in the case in which $\gamma_1=0$ and $\gamma_2\in [0,+\infty)$ remains an open question.

The paper is organized as follows: in Section \ref{section:setting} we set the problem and introduce the assumptions on the energy density. In Section \ref{section:compactness} we recall a few compactness results and the definition and some properties of multiscale convergence. 
Sections \ref{section:E} and \ref{section:liminf} are devoted to the identification of the limit linearized stress and to the proof of the liminf inequality \eqref{eq:intro-liminf}. In Section \ref{section:limsup} we show the optimality of the lower bound deduced in Section \ref{section:liminf}, and we exhibit a recovery sequence satisfying \eqref{eq:intro-limsup}.\\

\subsection{Notation}
In what follows, $Q:=\big(-\tfrac12,\tfrac12\big)^2$ denotes the unit cube in $\R^2$ centered at the origin and with sides parallel to the coordinate axes. We will write a point $x\in\R^3$ as
$$x=(x',x_3),\qquad\text{where }x'\in\R^2\text{ and }x_3\in\R,$$
and we will use the notation $\nabla'$ to denote the gradient with respect to $x'$. For every $r\in\R$, $\ceil{r}$ is its greatest integer part. With a slight abuse of notation, for every $x'\in\R^2$, $\ceil{x'}$ and $\floor{x'}$ are the points in $\R^2$ whose coordinates are given by the greatest and least integer parts of the coordinates of $x'$, respectively. Given a map $\phi\in W^{1,2}(\R^2)$, $(y\cdot\nabla')\phi(x')$ stands for
$$(y\cdot\nabla')\phi(x'):=y_1\partial_{x_1}\phi(x')+y_2\partial_{x_2}\phi(x')\quad\text{for a.e. }x'\in\R^2\text{ and }y\in Q.$$
We write $\np \phi$ to indicate the map
$$\np \phi(x'):=(-\partial_{x_2}\phi,\partial_{x_1}\phi)\quad\text{for a.e. }x'\in\R^2.$$ 

We denote by $\M^{n\times m}$ the set of matrices with $n$ rows and $m$ columns and by $SO(3)$ the set of proper rotations, that is
$$SO(3):=\{R\in\mthree:\,R^TR=Id\text{ and }{\rm{det }}\,R=1\}.$$
Given a matrix $M\in\mthree$, $M'$ stands for the $3\times 2$ submatrix of $M$ given by its first two columns. For every $M\in\M^{n\times n}$, ${\rm sym}\,M$ is the the $n\times n$ symmetrized matrix defined as $${\rm sym }\,M:=\frac{M+M^T}{2}.$$

Whenever a map $v\in L^2, C^{\infty},\cdots$, is $Q$-periodic, that is
$$v(x+e_i)=v(x)\quad i=1,2,$$
for a.e. $x\in \R^2$, where $\{e_1,e_2\}$ is the othonormal canonical basis of $\R^2$, we write $v\in L^2_{\rm per}, C^{\infty}_{\rm per},\cdots$, respectively. We implicitly identify the spaces $L^2(Q)$ and $L^2_{\rm per}(\R^2)$. We denote the Lebesgue measure of a measurable set $A\subset\R^N$ by $|A|$.

We adopt the convention that $C$ designates a generic constant, whose value may change from expression to expression in the same formula.
\section{Setting of the problem}
\label{section:setting}
Let $\omega\subset \R^2$ be a bounded Lipschitz domain whose boundary is piecewise $C^1$. This regularity assumption is only needed in Section \ref{section:limsup}, while the results in Sections \ref{section:compactness}--\ref{section:liminf} continue to hold for every bounded Lipschitz domain $\omega\subset\R^2$. We assume that the set
$$\Omega_h:=\omega\times(-\tfrac h2,\tfrac h2)$$
is the reference configuration of a nonlinearly elastic thin plate. In the sequel, $\{h\}$ and $\{\ep(h)\}$ are monotone decreasing sequences of positive numbers,  $h\to 0$, $\ep(h)\to 0$ as $h\to 0$, such that the following limits exist
$$\gamma_1:=\lim_{h\to 0}\frac{h}{\ep(h)}\quad\text{and}\quad\gamma_2:=\lim_{h\to 0}\frac{h}{\ep^2(h)},$$
with $\gamma_1,\gamma_2\in [0,+\infty]$. There are five possible regimes:
$\gamma_1,\gamma_2=+\infty$; $0<\gamma_1<+\infty$ and $\gamma_2=+\infty$; $\gamma_1=0$ and $\gamma_2=+\infty$; $\gamma_1=0$ and $0<\gamma_2<+\infty$;
$\gamma_1=0$ and $\gamma_2=0$. 
We focus here on the first three regimes, that is on the cases in which $\gamma_2=+\infty$. 

 For every deformation $v\in W^{1,2}(\Omega_h;\R^3)$, we consider its rescaled elastic energy 
 $$\mathcal{J}^h(v):=\frac{1}{h}\int_{\Omega_h}W\Big(\frac{x'}{\ep(h)}, \frac{x'}{\ep^2(h)}, \nabla v(x)\Big)\,dx,$$
 where $W:\R^2\times\R^2\times\M^{3\times 3}\to [0,+\infty)$ represents the stored energy density of the plate, and $(y,z,F)\mapsto W(y,z,F)$ is measurable and $Q$-periodic in its first two variables, i.e., with respect to $y$ and $z$. We also assume that for a.e. $y$ and $z$, the map $W(y,z,\cdot)$ is continuous and satisfies the following assumptions:
 \begin{enumerate}[({H}1)]
 \item $W(y,z,RF)=W(y,z,F)$ for every $F\in\mthree$ and for all $R\in SO(3)$ (frame indifference),
 \item $W(y,z,F)\geq C_1\,\dist^2(F;SO(3))$ for every $F\in\mthree$ (nondegeneracy),
 \item there exists $\delta>0$ such that $W(y,z,F)\leq C_2\,\dist^2(F;SO(3))$ for every $F\in\mthree$ with $\dist(F;SO(3))<\delta$,
 \item $\lim_{|G|\to 0}\frac{W(y,z,Id+G)-{\mathscr{Q}}(y,z,G)}{|G|^2}=0$, where ${\mathscr{Q}}(y,z,\cdot)$ is a quadratic form on $\mthree$.
  \end{enumerate}
    By assumptions (H1)--(H4) we obtain the following lemma, which guarantees the continuity of the quadratic map ${\mathscr{Q}}$ introduced in (H4). 
\begin{lemma}
\label{lemma:Q-continuous}
Let $W:\R^2\times\R^2\times\M^{3\times 3}\to [0,+\infty)$ satisfy (H1)--(H4) and let ${\mathscr{Q}}:\R^2\times\R^2\times\mthree\to[0,+\infty)$ be defined as in (H4). Then,
\begin{enumerate}
\item[(i)] ${\mathscr{Q}}(y,z,\cdot)$ is continuous for a.e. $y,z\in \R^2$,
\item[(ii)] ${\mathscr{Q}}(\cdot,\cdot,F)$ is $Q\times Q$-periodic and measurable for every $F\in\mthree$,
\item[(iii)] for a.e. $y,z\in\R^2$, the map ${\mathscr{Q}}(y,z,\cdot)$ is quadratic on $\mthree_{\rm sym}$, and satisfies
$$\frac{1}{C}|{\rm sym F}|^2\leq {\mathscr{Q}}(y,z,F)={\mathscr{Q}}(y,z,{\rm symF})\leq C|{\rm sym F}|^2$$
for all $F\in\mthree$, and some $C>0$.
In addition, there exists a monotone function 
$$r:[0,+\infty)\to[0,+\infty],$$ such that $r(\delta)\to 0$ as $\delta\to 0$, and
$$|W(y,z,Id+F)-{\mathscr{Q}}(y,z,F)|\leq |F|^2r(|F|)$$
for all $F\in\mthree$, for a.e. $y,z\in\R^2$.
\end{enumerate}
\end{lemma}
We refer to \cite[Lemma 2.7]{neukamm}  and to \cite[Lemma 4.1]{neukamm.velcic} for a proof of Lemma \ref{lemma:Q-continuous} in the case in which ${\mathscr{Q}}$ is independent of $z$. The proof in the our setting is a straightforward adaptation.

  As it is usual in dimension reduction analysis, we perform a change of variables in order to reformulate the problem on a domain independent of the varying thickness parameter. We set $$\Omega:=\Omega_1=\omega\times(-\tfrac12,\tfrac12)$$
  and we consider the change of variables $\psi^h:\Omega\to \Omega^h$, defined as 
  $$\psi^h(x)=(x',h x_3)\quad\text{for every }x\in\Omega.$$To every deformation $v\in W^{1,2}(\Omega_h;\R^3)$ we associate a function $u\in W^{1,2}(\Omega;\R^3)$, defined as $u:=v\circ \psi^h$, whose elastic energy is given by
  $$\mathcal{E}^h(u)=\mathcal{J}^h(v)=\int_{\Omega}W\Big(\frac{x'}{\ep(h)}, \frac{x'}{\ep^2(h)}, \nh u(x)\Big)\,dx,$$
  where
  $$\nh u(x):=\Big(\nabla'u(x)\big|\frac{\partial_{x_3}u(x)}{h}\Big)\quad\text{for a.e. }x\in\Omega.$$
  
  In this paper we focus on the asymptotic behavior of sequences of deformations $\{u^h\}\subset W^{1,2}(\Omega_h;\R^3)$ satisfying the uniform energy estimate
  \be{eq:uniform-en-estimate}
  \mathcal{E}^h(u^h):=\int_{\Omega}W\Big(\frac{x'}{\ep(h)}, \frac{x'}{\ep^2(h)}, \nh u^h(x)\Big)\,dx\leq Ch^2\quad\text{for every }h>0.
  \ee
  We remark that in the case in which $W$ is independent of $y$ and $z$, such scalings of the energy lead to Kirchhoff's nonlinear plate theory, which was rigorously justified by means of $\Gamma-$convergence tecniques in the seminal paper \cite{friesecke.james.muller}.
  
\section{Compactness results and multiscale convergence}
\label{section:compactness}
In this section we present a few preliminary results which will allow us to deduce compactness for sequences of deformations satisfying the uniform energy estimate \eqref{eq:uniform-en-estimate}.  

We first recall \cite[Theorem 4.1]{friesecke.james.muller}, which provides a characterization of limits of deformations whose scaled gradients are uniformly close in the $L^2$-norm to the set of proper rotations.
\begin{theorem}
\label{thm:compactness-def}
Let $\{u^h\}\subset W^{1,2}(\Omega;\R^3)$ be such that
\be{eq:unif-bending-en}
\limsup_{h\to 0} \frac{1}{h^2}\iO \dist^2(\nh u^h(x),SO(3))\,dx<+\infty.
\ee
Then, there exists a map $u\in W^{2,2}(\omega;\R^3)$ such that, up to the extraction of a (not relabeled) subsequence,
\begin{align}
\nonumber&u^h-\fint_{\Omega} u^h(x)\,dx\to u\quad\text{strongly in }L^2(\Omega;\R^3)\\
\nonumber&\nh u^h\to (\nabla'u|n_u)\quad\text{strongly in }L^2(\Omega;\mthree),
\end{align}
with
\be{eq:normal1}\partial_{x_\alpha} u(x')\cdot\partial_{x_\beta} u(x')=\delta_{\alpha,\beta}\quad\text{for a.e. }x'\in\omega,\quad\,\alpha,\beta\in\{1,2\} \ee
and
\be{eq:normal2}n_u(x'):=\partial_{x_1}u(x')\wedge\partial_{x_2}u(x')\quad\text{for a.e. }x'\in\omega.\ee
\end{theorem}
A crucial point in the proof of the liminf inequality \eqref{eq:intro-liminf}(see Sections \ref{section:E} and \ref{section:liminf}) is to approximate the scaled gradients of deformations with uniformly small energies, by sequences of maps which are either piecewise constant on cubes of size comparable to the homogenization parameters with values in the set of proper rotations, or have Sobolev regularity and are close in the $L^2$-norm to piecewise constant rotations. The following lemma has been stated in \cite[Lemma 3.3]{velcic}, and its proof follows by combining \cite[Theorem 6]{friesecke.james.muller2} with the argument in \cite[Proof of Theorem 4.1, and Section 3]{friesecke.james.muller}. We remark that the additional regularity of the limit deformation $u$ in Theorem \ref{thm:compactness-def} is a consequence of Lemma \ref{lemma:rotations1}, and in particular of the approximation of scaled gradients by $W^{1,2}$ maps. 
\begin{lemma}
\label{lemma:rotations1}
Let $\gamma_0\in(0,1]$ and let $h,\delta>0$ be such that
$$\gamma_0\leq\frac{h}{\delta}\leq\frac{1}{\gamma_0}.$$
There exists a constant $C$, depending only on $\omega$ and $\gamma_0$, such that for every $u\in W^{1,2}(\omega;\R^3)$ there exists a map $R:\omega\to SO(3)$ piecewise constant on each cube $x+\delta Y$, with $x\in\delta\Z^2$, and there exists $\tilde{R}\in W^{1,2}(\omega;\mthree)$ such that 
\begin{align*}
&\|\nh u-R\|_{L^2(\Omega;\mthree)}^2+\|R-\tilde{R}\|_{L^2(\omega;\mthree)}^2
+h^2\|\nabla ' \tilde{R}\|_{L^2(\omega;\mthree\times\mthree)}^2\\
&\quad\leq C\|\dist(\nh u;SO(3))\|_{L^2(\Omega)}.
\end{align*}
Moreover, for every $\xi\in \R^2$ satisfying 
$$|\xi|_{\infty}:=\max\{|\xi\cdot e_1|,|\xi\cdot e_2|\}<h,$$
and for every $\omega'\subset\omega$, with $\dist(\omega',\partial\omega)>Ch$, there holds
$$\|R(x')-R(x'+\xi)\|_{L^2(\omega';\mthree)}\leq C\|\dist(\nh u;SO(3))\|_{L^2(\Omega)}.$$
\end{lemma}

We now recall the definitions of ``2-scale convergence" and ``3-scale convergence". For a detailed treatment of two-scale convergence we refer to, \emph{e.g.}, \cite{allaire, lukkassen.nguetseng.wall, nguetseng}. The main results on multiscale convergence may be found in  \cite{allaire.briane, baia.fonseca, ferreira.fonseca, ferreira.fonseca2}.\begin{definition}
Let $D$ be an open set in $\R^N$ and let $Y^N$ be the unit cube in $\R^N$,
$$Y^N:=\Big(-\tfrac12,\tfrac12\Big)^N.$$
Let $u\in L^2(D\times Y^N)$ and $\{u^h\}\in L^2(D)$. We say that $\{u^h\}$ \emph{converges weakly 2-scale to $u$} in $L^2(D\times Y^N)$, and we write $u^h\swkts u$  if
$$\int_{D} u^h(\xi)\varphi\Big(\xi,\frac{\xi}{\ep(h)} \Big)\,d\xi\to \int_{D}\int_{Y^N} {u(\xi,\eta)\varphi(\xi,\eta)}\,d\eta\,d\xi$$
for every $\varphi\in C^{\infty}_c(D;C_{\rm per}(Y^N))$.\\

Let $u\in L^2(D\times Y^N\times Y^N)$ and $\{u^h\}\in L^2(D)$. We say that $\{u^h\}$ \emph{converges weakly 3-scale to $u$} in $L^2(D\times Y^N\times Y^N)$, and we write $u^h\swks u$, if
$$\int_{D} u^h(\xi)\varphi\Big(\xi,\frac{\xi}{\ep(h)}, \frac{\xi}{\ep^2(h)} \Big)\,d\xi\to  \int_{D}\int_{Y^N}\int_{Y^N}{u(\xi,\eta,\lambda)\varphi(\xi,\eta,\lambda)}\,d\lambda\,d\eta\,d\xi$$
for every $\varphi\in C^{\infty}_c(D;C_{\rm per}(Y^N\times Y^N))$. 

We say that $\{u^h\}$ \emph{converges strongly 3-scale to $u$} in $L^2(D\times Y^N\times Y^N)$, and we write $u^h\sstts u$, if
$$u^h\swks u\quad\text{weakly 3-scale}$$
and $$\|u^h\|_{L^2(D)}\to \|u\|_{L^2(D\times Y^N\times Y^N)}.$$
\end{definition}
In order to simplify the statement of Theorem \ref{thm:limit stresses} and its proof, we introduce the definition of ``dr-3-scale convergence" (dimension reduction three-scale convergence), i.e., 3-scale convergence adapted to dimension reduction, inspired by S. Neukamm's 2-scale convergence adapted to dimension reduction (see \cite{neukammPHD}).
\begin{definition}
\label{def:dr}
Let $u\in L^2(\Omega\times Q\times Q)$ and $\{u^h\}\in L^2(\Omega)$. We say that $\{u^h\}$ \emph{converges weakly dr-3-scale to $u$} in $L^2(\Omega\times Q\times Q)$, and we write $u^h\wks u$, if
$$\iO u^h(x)\varphi\Big(x,\frac{x'}{\ep(h)}, \frac{x'}{\ep^2(h)} \Big)\,dx\to \intoyz {u(x,y,z)\varphi(x,y,z)}$$
for every $\varphi\in C^{\infty}_c(\Omega;C_{\rm per}(Q\times Q))$. 
\end{definition}
\begin{remark}
We point out that ``dr-3-scale convergence" is just a particular case of classical 3-scale convergence. Indeed, what sets apart ``dr-3-scale convergence" from the classical 3-scale convergence is solely the fact that the test functions in Definition \ref{def:dr} depend on $x_3$ but oscillate only in the cross-section $\omega$. In particular, if $\{u^h\}\in L^2(\Omega)$ and
$$u^h\wks u\quad\text{weakly dr 3 scale}$$
then $\{u^h\}$ is bounded in $L^2(\Omega)$. Therefore, by \cite[Theorem 1.1]{allaire.briane} there exists $\xi\in L^2(\Omega\times (Q\times (-\tfrac12,\tfrac12))\times (Q\times (-\tfrac12,\tfrac12)))$ such that, up to the extraction of a (not relabeled) subsequence,
$$u^h\overset{3-s}{\lrightharpoonup}\xi\quad\text{weakly 3-scale },$$
that is $u^h$ weakly 3-scale converge to $\xi$ in $L^2(\Omega\times (Q\times (-\tfrac12,\tfrac12))\times (Q\times (-\tfrac12,\tfrac12)))$ (in the sense of classical 3-scale convergence). Hence, the ``dr-3-scale limit" $u$ and the ``classical 3-scale limit" $\xi$ are related by
$$u(x,y,z)=\int_{-\tfrac12}^{\tfrac12}\int_{-\tfrac12}^{\tfrac12}\xi(x,y,z,\eta,\tau)\,d\eta\,d\tau\quad\text{for a.e. }x\in\omega\text{ and }y,z\in Q.$$
\end{remark}
We now state a theorem regarding the characterization of limits of scaled gradients in the multiscale setting adapted to dimension reduction. We omit its proof as it is a simple generalization of the arguments in \cite[Theorem 6.3.3]{neukammPHD}.

\begin{theorem}
\label{thm:mod3scales-limits}
Let $u,\{u^h\}\subset W^{1,2}(\Omega)$ be such that
$$u^h\wk u\quad\text{weakly in }W^{1,2}(\Omega).$$
and
$$\limsup_{h\to 0}\iO |\nh u^h(x)|^2\,dx<\infty.$$
Then $u$ is independent of $x_3$. Moreover, there exist $u_1\in L^2(\Omega;W^{1,2}_{\rm per}(Q))$, $u_2\in L^2(\Omega\times Q;W^{1,2}_{\rm per}(Q))$, and $\bar{u}\in L^2\big(\omega\times Q\times Q;W^{1,2}\big(-\tfrac12,\tfrac12\big)\big)$ such that, up to the extraction of a (not relabeled) subsequence,
$$\nh u^h\wks\big(\nabla'u+\nabla_{y}u_1+\nabla_{z}u_2\Big|\partial_{x_3}\bar{u}\big)\quad\text{weakly dr-3-scale}.$$

Moreover,
\begin{enumerate}[(i)]
\item if $\gamma_1=\gamma_2=+\infty$ (\emph{i.e.} $\ep(h)<<h$), then $\partial_{y_i}\bar{u}=\partial_{z_i}\bar{u}=0$, for $i=1,2$;
\item if $0<\gamma_1<+\infty$ and $\gamma_2=+\infty$ (\emph{i.e.} $\ep(h)\sim h$), then
$$\bar{u}=\frac{u_1}{\gamma_1};$$
\item if $\gamma_1=0$ and $\gamma_2=+\infty$ (\emph{i.e.} $h<<\ep(h)<<h^{\tfrac12}$), then
$$\partial_{x_3}u_1=0\quad\text{and}\quad\partial_{z_i}\bar{u}=0,\,i=1,2.$$
\end{enumerate}
\end{theorem}

In the last part of this section we collect some properties of sequences having unbounded $L^2$ norms but whose oscillations on the scale $\ep$ or $\ep^2$ are uniformly controlled. Arguing as in \cite[Lemmas 3.6--3.8]{hornung.neukamm.velcic}, we highlight the multi-scale oscillatory behavior of our sequences by testing them against products of maps with compact support and oscillatory functions with vanishing average in their periodicity cell. 
In the proof of Theorem \ref{thm:limit stresses} we refer to \cite[Proposition 3.2]{hornung.neukamm.velcic} and \cite[Proposition 3.2]{velcic}, so for simplicity we introduce the notation needed in those papers.
 \begin{definition}
 \label{def:oscillatory}
 Let $\tilde{f}\in L^2(\omega\times Q)$ be such that
 $$\int_Q\tilde{f}(\cdot,y)\,dy=0\quad\text{a.e. in }\omega.$$
 We write
 \emph{$$f^h\oscy \tilde{f}$$}
 if 
 $$\lim_{h\to 0} \int_{\omega} f^h(x')\varphi(x')g\Big(\frac{x'}{\ep(h)}\Big)\,dx'=\int_{\omega}\int_{Q}{\tilde{f}(x',y)\varphi(x')g(y)}\,dy\,dx'$$
 for every $\varphi\in C^{\infty}_c(\omega)$ and $g\in C^{\infty}_{\rm per}(Q)$, with $\int_Q g(y)\,dy=0$.\\
 
\medskip

 Let $\{f^h\}\subset L^2(\omega)$ and let $\tilde{\tilde{f}}\in L^2(\omega\times Q\times Q)$ be such that
 $$\int_Q\tilde{\tilde{f}}(\cdot,\cdot,z)\,dz=0\quad\text{a.e. in }\omega\times Q.$$ We write
 \emph{$$f^h\oscz \tilde{\tilde{f}}$$}
 if 
 $$\lim_{h\to 0} \int_{\omega} f^h(x')\psi\Big(x',\frac{x'}{\ep(h)}\Big)\varphi\Big(\frac{x'}{\ep^2(h)}\Big)\,dx'=\intomyz{\tilde{\tilde{f}}(x',y,z)\psi(x',y)\varphi(z)}$$
 for every $\psi\in C^{\infty}_c(\omega; C^{\infty}_{\rm per}(Q))$ and $\varphi\in C^{\infty}_{\rm per}(Q)$, with $\int_Q \varphi(z)\,dz=0$.\\
 
  \medskip
 
 \end{definition}
 {\begin{remark}
 \label{remark:osc}
 As a direct consequence of the definition of multiscale convergence and density arguments, if $\{f^h\}\subset L^2(\omega)$, then
 $$f^h\swkts f\quad\text{weakly 2-scale}$$
if and only if
 $$f^h(x)\oscy f(x)-\iQ f(x,y)\,dy.$$
 Analogously, 
 $$f^h\swks \tilde{f}\quad\text{weakly 3-scale}$$
if and only if
 $$f^h(x)\oscz \tilde{f}-\iQ \tilde{f}(x,y,z)\,dz.$$
 \end{remark}}

 We recall finally \cite[Lemma 3.7 and Lemma 3.8]{hornung.neukamm.velcic}.
 \begin{lemma}
 \label{lemma:pw-osc}
 Let $\{f^h\}\subset L^{\infty}(\omega)$ and $f^0\in L^{\infty}(\omega)$ be such that 
 $$f^h\overset{*}{\wk}f^0\quad\text{weakly-* in }L^{\infty}(\omega).$$
 Assume that $f^h$ are constant on each cube $Q(\ep(h)z,\ep(h))$, with $z\in\Z^2$. If $f^0\in W^{1,2}(\omega)$, then
 $$\frac{f^h}{\ep(h)}\oscy-(y\cdot\nabla')f^0.$$
 \end{lemma}
 
 \begin{lemma}
 \label{lemma:sobolev-osc}
 Let $\{f^h\}\subset W^{1,2}(\omega)$, $f^0\in W^{1,2}(\omega)$, and $\phi\in L^2(\omega;W^{1,2}_{\rm per}(Q))$ be such that 
 $$f^h\wk f^0\quad\text{weakly in }W^{1,2}(\omega),$$
 and 
 $$\nabla'f^h\swkts\nabla'f^0+\nabla_y\phi \quad\text{weakly 2-scale},$$
 with $\iQ\phi(x',y)\,dy=0$ for a.e. $x'\in\omega$. Then, 
 $$\frac{f^h}{\ep(h)}\oscy\phi.$$
 \end{lemma}

\section{Identification of the limit stresses}
\label{section:E}

Due to the linearized behavior of the nonlinear elastic energy around the set of proper rotations, a key point in the proof of the liminf inequality \eqref{eq:intro-liminf} is to establish a characterization of the weak limit, in the sense of 3-scale-dr convergence, of the sequence of linearized elastic stresses
$$E^h:=\frac{\sqrt{(\nh u^h)^T\nh u^h}-Id}{h}.$$
We introduce the following classes of functions:
\begin{align}
\label{eq:c1inf}
\mathcal{C}_{\gamma_1,+\infty}&:=\Big\{U\in L^2(\Omega\times Q\times Q;\mthree):\\
\nonumber &\text{there exists } \phi_1\in L^2\big(\omega; W^{1,2}\big(\big(-\tfrac12,\tfrac12\big);W^{1,2}_{\rm per}(Q;\R^3))\big)\\
\nonumber &\text{and }\phi_2\in L^2(\Omega\times Q;W^{1,2}_{\rm per}(Q;\R^3))\\
\nonumber &\text{such that }
U={\rm sym }\Big(\nabla_y \phi_1\Big|\frac{\partial_{x_3}\phi_1}{\gamma_1}\Big)+{\rm sym }(\nabla_z\phi_2|0)\Big\},
\end{align}
\begin{align}
\label{equation:cinfinf}
\mathcal{C}_{+\infty,+\infty}&:=\Big\{U\in L^2(\Omega\times Q\times Q;\mthree):\\
\nonumber &\text{there exists } d\in L^2(\Omega;\R^3),\,\phi_1\in L^2\big(\Omega;W^{1,2}_{\rm per}(Q;\R^3)\big)\\
\nonumber &\text{and }\phi_2\in L^2(\Omega\times Q;W^{1,2}_{\rm per}(Q;\R^3))\\
\nonumber &\text{such that }
U={\rm sym }(\nabla_y \phi_1|d)+{\rm sym }(\nabla_z\phi_2|0)\Big\},
\end{align}
and
\begin{align}
\label{eq:c0inf}
\mathcal{C}_{0,+\infty}&:=\Big\{U\in L^2(\Omega\times Q\times Q;\mthree):\\
\nonumber &\text{there exists } \xi\in L^2(\Omega;W^{1,2}_{\rm per}(Q;\R^2)),\,\eta\in L^2(\omega;W^{2,2}_{\rm per}(Q)),\\
\nonumber&g_i\in L^2(\Omega\times Y),\,i=1,2,3,\,\text{ and }\phi\in L^2(\Omega\times Q;W^{1,2}_{\rm per}(Q;\R^3))\text{ such that }\\
\nonumber &U={\rm sym }\Bigg(\begin{array}{cc}\nabla_y\xi+x_3\nabla_y^2\eta&g_1\\
\phantom{g}&g_2\\
g_1\quad g_2&g_3\end{array}\Big)+{\rm sym }(\nabla_z\phi|0)\Bigg\}.
\end{align}

We now state the main result of this section.
\begin{theorem}
\label{thm:limit stresses}
Let $\gamma_1\in[0,+\infty]$ and $\gamma_2=+\infty$. Let $\{u^h\}\subset W^{1,2}(\Omega;\R^3)$ be a sequence of deformations satisfying \eqref{eq:unif-bending-en} and converging to a deformation $u$ in the sense of Theorem \ref{thm:compactness-def}. Then there exist 
\mbox{$E\in L^2(\Omega\times Q\times Q;\mthree_{\rm sym})$}, $B\in L^2(\omega;\M^{2\times 2})$, and $U\in C_{\gamma_1, {+\infty}}$, such that, up to the extraction of a (not relabeled) subsequence,
$$E^h\wks E\quad\text{weakly dr-3-scale},$$
where
$$E(x,y,z)=\Big(\begin{array}{cc}x_3\Pi^u(x')+{\rm sym\,}B(x')&0\\
0&0\end{array}\Big)+U(x,y,z),$$
for almost every $(x,y,z)\in\Omega\times Q\times Q$, with
\be{eq:def-Pi}\Pi_{\alpha,\beta}^u(x'):=-\partial^2_{\alpha,\beta}u(x')\cdot n_u(x')\,\text{for }\alpha,\beta=1,2,\ee
and $n_u(x'):=\partial_1 u(x')\wedge\partial_2 u(x')$ for every $x'\in\omega$.
\end{theorem}
\begin{proof}
Let $\{u^h\}$ be as in the statement of the theorem. By Theorem \ref{thm:compactness-def} the map $u\in W^{2,2}(\omega;\R^3)$ is an isometry, and
\be{eq:scaled-grad-conv}\nh u^h\to (\nabla'u|n_u)\quad\text{strongly in }L^2(\Omega;\mthree).\ee
For simplicity, we subdivide the proof into three cases, corresponding to the three regimes $0<\gamma_1<+\infty$, $\gamma_1=+\infty$, and $\gamma_1=0$, and each case will be treated in multiple steps.

\medskip

\noindent\textbf{Case 1: $0<\gamma_1<+\infty$ and $\gamma_2=+\infty$}.\\
Applying Lemma \ref{lemma:rotations1} with $\delta(h)=\ep(h)$, we construct two sequences $\{R^h\}\subset L^{\infty}(\omega; SO(3))$ and $\{\tilde{R}^h\}\subset W^{1,2}(\omega;\mthree)$ such that $R^h$ is piecewise constant on every cube of the form $Q(\ep(h)z,\ep(h))$, with $z\in \Z^2$, and 
\begin{align}
\label{eq:approximate-rot}
&\|\nh u^h-R^h\|_{L^2(\Omega;\mthree)}^2+\|R^h-\tilde{R}^h\|_{L^2(\omega;\mthree)}^2\\
\nonumber &\quad+h^2\|\nabla' \tilde{R}^h\|_{L^2(\omega;\mthree\times\mthree)}^2\leq C\|\dist(\nh u^h; SO(3))\|_{L^2(\Omega)}^2.
\end{align}
By \eqref{eq:unif-bending-en} and \eqref{eq:approximate-rot}, there holds
$$\nh u^h-R^h\to 0\quad\text{strongly in }L^2(\Omega;\mthree),$$
$$R^h-\tilde{R}^h\to 0\quad\text{strongly in }L^2(\Omega;\mthree),$$
and $\{\tilde{R}^h\}$ is bounded in $W^{1,2}(\omega;\mthree)$.
Therefore, by \eqref{eq:scaled-grad-conv} and the uniform boundedness of the sequence $\{R^h\}$ in $L^{\infty}(\omega;\mthree)$, and in particular in $L^2(\omega;\mthree)$,
\be{eq:pwconst-conv}R^h\to R\quad\text{strongly in }L^2(\omega;\mthree),\qquad R^h\wk^* R\quad\text{weakly* in }L^{\infty}(\omega;\mthree),\ee
and
\be{eq:sobolev-conv}\tilde{R}^h\wk R\quad\text{weakly in }W^{1,2}(\omega;\mthree),\ee
where
\be{eq:def-lim-R}R:=(\nabla'u|n_u).\ee

In order to identify the multiscale limit of the linearized stresses, we argue as in \cite[Proof of Proposition 3.2]{hornung.neukamm.velcic}, and we introduce the scaled linearized strains 
\be{eq:Gh}G^h:=\frac{(R^h)^T\nh u^h-Id}{h}.\ee
By \eqref{eq:unif-bending-en} and \eqref{eq:approximate-rot} the sequence $\{G^h\}$ is uniformly bounded in $L^2(\Omega;\mthree)$. By standard properties of 3-scale convergence (see \cite[Theorem 2.4]{allaire.briane}) there exists $G\in L^2(\Omega\times Q\times Q;\mthree)$ such that, up to the extraction of a (not relabeled) subsequence,
\be{eq:gh-3s}G^h\swks G\quad\text{weakly 3-scale}.\ee
By the identity
$$\sqrt{(Id+hF)^T(Id+hF)}=Id+h\,{\rm{sym}}\,F+O(h^2),$$
and observing that
$$E^h=\frac{\sqrt{(\nh u^h)^T\nh u^h}-Id}{h}=\frac{\sqrt{(Id+hG^h)^T(Id+hG^h)}-Id}{h},$$ there holds 
\be{eq:E-G} E=\rm{sym}\,G.\ee
 
By \eqref{eq:gh-3s}, it follows that
$$G^h\swkts\iQ G(x,y,z)\,dz\quad\text{weakly 2-scale.}$$
Therefore, by \cite[Proposition 3.2]{hornung.neukamm.velcic} there exist $B\in L^2(\omega;\M^{2\times 2})$ and $\phi_1\in$\newline $L^2\big(\omega; 
W^{1,2}\big(\big(-\tfrac12,\tfrac12\big);W^{1,2}_{\rm per}(Q;\R^3)\big)\big)$ such that
\begin{align}
\label{eq:first-part-E-g1}
&{\rm sym}\,\iQ G(x,y,\xi)\,d\xi\\
\nonumber&\quad=\Big(\begin{array}{cc}x_3\Pi^u(x')+{\rm sym }\,B(x')&0\\
0&0\end{array}\Big)+{\rm sym }\left(\nabla_y \phi_1(x,y)\Big|\frac{\partial_{x_3}\phi_1(x,y)}{\gamma_1}\right)
\end{align}
for a.e. $x\in\Omega$ and $y\in Y$. Thus, by \eqref{eq:E-G} and \eqref{eq:first-part-E-g1} to complete the proof we only need to prove that
\be{eq:remaining-E-gamma_1}{\rm{sym }}\,G(x,y,z)-{\rm{sym }}\,\iQ G(x,y,\xi)\,d\xi={\rm sym}\,(\nabla_z\phi_2(x,y,z)|0)\ee
for some $\phi_2\in L^2(\Omega\times Q;W^{1,2}_{\rm per}(Q;\R^3)).$

Set \be{eq:def-uh}\bar{u}^h(x'):=\int_{-\tfrac12}^{\tfrac12}u^h(x',x_3)\,dx_3\quad\text{for a.e. }x'\in\omega\ee
and define $r^h\in W^{1,2}(\Omega;\R^3)$ as
\be{eq:def-rh}u^h(x)=:\bar{u}^h(x')+hx_3\tilde{R}^h(x')e_3+hr^h(x',x_3)\quad\text{for a.e. }x\in\Omega.\ee
We remark that 
\be{eq:zh-null-average}\int_{-\tfrac12}^{\tfrac12}r^h(x',x_3)\,dx_3=0,\ee and
\be{eq:scaled-grad-dec}
\frac{\nh u^h-R^h}{h}=\Big(\frac{\nabla' \bar{u}^h-(R^h)'}{h}+x_3\nabla' \tilde{R}^h e_3\Big|\frac{(\tilde{R}^h-R^h)}{h}e_3\Big)+\nh r^h.
\ee

We first notice that by \eqref{eq:unif-bending-en}, \eqref{eq:approximate-rot}, \eqref{eq:sobolev-conv}, and \eqref{eq:zh-null-average}, the sequence $\{r^h\}$ is uniformly bounded in $W^{1,2}(\Omega;\R^3)$. Hence, by Theorem \ref{thm:mod3scales-limits} (ii) there exist $r\in W^{1,2}(\omega;\R^3)$, $\hat{\phi}_1\in L^2\big(\omega; W^{1,2}\big(\big(-\tfrac12,\tfrac12\big);W^{1,2}_{\rm per}(Q;\R^3)\big)\big)$ and $\hat{\phi}_2\in L^2(\Omega\times Q;$ $W^{1,2}_{\rm per}(Q;\R^3))$ such that, up to the extraction of a (not relabeled) subsequence,
\be{eq:rh-lim}\nh r^h\wks \Big(\nabla' r+\nabla_y \hat{\phi}_1+\nabla_z\hat{\phi}_2\Big|\frac{\partial_{x_3}\hat{\phi}_1}{\gamma_1}\Big)\quad\text{weakly dr-3-scale}.\ee
By \eqref{eq:unif-bending-en} and \eqref{eq:approximate-rot}, and since $R^h$ does not depend on $x_3$, $\big\{\frac{\nh\bar{u}^h-(R^h)'}{h}\big\}$ is bounded in $L^2(\omega;\M^{3\times 2})$. Therefore by \cite[Theorem 2.4]{allaire.briane} there exists $V\in L^2(\omega\times Q\times Q;\M^{3\times 2})$ such that, up to the extraction of a (not relabeled) subsequence,
\be{eq:V-dr-3}\frac{\nabla' \bar{u}^h-(R^h)'}{h}\swks V\quad\text{weakly 3-scale}.\ee

\noindent\emph{Case 1, Step 1: Characterization of V.}\\
In view of \eqref{eq:remaining-E-gamma_1}, we provide a characterization of
$$V(x',y,z)-\iQ V(x',y,\xi)\,d\xi.$$
 We claim that there exists $v\in L^2(\omega\times Q; W^{1,2}_{\rm per}(Q;\R^3))$ such that
\be{eq:claim-V2}
V(x',y,z)-\int_Q V(x',y,\xi)\,d\xi=\nabla_zv(x',y,z)\quad\text{for a.e. }x'\in\omega,\text{ and }y,z\in Q.
\ee
Arguing as in \cite[Proof of Proposition 3.2]{hornung.neukamm.velcic}, we first notice that by \cite[Lemma 3.7]{allaire.briane} to prove \eqref{eq:claim-V2} it is enough to show that
\be{eq:ort-z}
\intomyz {\Big(V(x',y,z)-\iQ V(x',y,\xi)\,d\xi\Big):\np \varphi(z)\psi(x',y)}=0
\ee
for every $\varphi\in C^1_{\rm per}(Q;\R^3)$ and $\psi\in C^{\infty}_c(\omega;C^{\infty}_{\rm per}(Q))$. 
Fix $\varphi\in C^1_{\rm per}(Q;\R^3)$ and $\psi\in C^{\infty}_c(\omega;C^{\infty}_{\rm per}(Q))$. We set 
$$\tilde{\varphi}^{\ep}(x'):=\ep^2(h)\varphi\Big(\frac{x'}{\ep^2(h)}\Big)\quad\text{for every }x'\in\omega.$$
Then,
\begin{align}
\label{eq:expr}
&\int_{\omega}\frac{\nabla'\bar{u}^h(x')}{h}:\np \varphi\Big(\frac{x'}{\ep^2(h)}\Big)\psi\Big(x',\frac{x'}{\ep(h)}\Big)\,dx'\\
\nonumber &\quad=\int_{\omega}\frac{\nabla'\bar{u}^h(x')}{h}:(\nabla')^{\perp} {\tilde{\varphi}}^{\ep}(x')\psi\Big(x',\frac{x'}{\ep(h)}\Big)\,dx'\\
\nonumber &\quad=\int_{\omega}\frac{\nabla'\bar{u}^h(x')}{h}:(\nabla')^{\perp} \Big[{\tilde{\varphi}}^{\ep}(x')\psi\Big(x',\frac{x'}{\ep(h)}\Big)\Big]\,dx'\\
\nonumber &\qquad-\int_{\omega}\frac{\nabla'\bar{u}^h(x')}{h}:\left[{\tilde{\varphi}}^{\ep}(x')\otimes\left((\nabla')_x^{\perp} \psi\Big(x',\frac{x'}{\ep(h)}\Big)+\frac{1}{\ep(h)}(\nabla')_y^{\perp} \psi\Big(x',\frac{x'}{\ep(h)}\Big)\right)\right]\,dx'.
\end{align}
The first term in the right-hand side of \eqref{eq:expr} is equal to zero, due to the definition of $\np$. Therefore we obtain
\begin{align}
\label{eq:almost-final}
&\int_{\omega}\frac{\nabla'\bar{u}^h(x')}{h}:\np \varphi\Big(\frac{x'}{\ep^2(h)}\Big)\psi\Big(x',\frac{x'}{\ep(h)}\Big)\,dx'\\
\nonumber&\quad=-\frac{\ep^2(h)}{h}\int_{\omega}\nabla'\bar{u}^h(x'):\Big[\varphi\Big(\frac{x'}{\ep^2(h)}\Big)\otimes(\nabla')^{\perp}_x\psi\Big(x',\frac{x'}{\ep(h)}\Big)\Big]\\
\nonumber&\qquad-\frac{\ep(h)}{h}\int_{\omega}\nabla'\bar{u}^h(x'):\Big[\varphi\Big(\frac{x'}{\ep^2(h)}\Big)\otimes(\nabla')_y^{\perp} \psi\Big(x',\frac{x'}{\ep(h)}\Big)\Big].
\end{align}
 By \eqref{eq:approximate-rot}, the regularity of the test functions, and since $\gamma_2=+\infty$, we get
\be{eq:almost-final1}\frac{\ep^2(h)}{h}\int_{\omega}\nabla'\bar{u}^h(x'):\Big[\varphi\Big(\frac{x'}{\ep^2(h)}\Big)\otimes(\nabla')^{\perp}_x\psi\Big(x',\frac{x'}{\ep(h)}\Big)\Big]\,dx'\to 0,\ee
while by \eqref{eq:scaled-grad-conv}, \eqref{eq:def-lim-R}, and the regularity of the test functions,
\begin{align}
\label{eq:almost-final2}
&\lim_{h\to 0}\frac{\ep(h)}{h}\int_{\omega}\nabla'\bar{u}^h(x'):\Big[\varphi\Big(\frac{x'}{\ep^2(h)}\Big)\otimes(\nabla')_y^{\perp} \psi\Big(x',\frac{x'}{\ep(h)}\Big)\Big]\,dx'\\
\nonumber&\quad=\frac{1}{\gamma_1}\intomyz{R'(x'):(\varphi(z)\otimes(\nabla')_y^{\perp} \psi(x',y))}=0,
\end{align}
where the latter equality is due to the periodicity of $\psi$ with respect to the $y$ variable.
Combining \eqref{eq:expr}, \eqref{eq:almost-final}, \eqref{eq:almost-final1} and \eqref{eq:almost-final2}, we conclude that
\be{eq:claim2-part1}\lim_{h\to 0}\int_{\omega}\frac{\nabla'\bar{u}^h(x')}{h}:\np \varphi\Big(\frac{x'}{\ep^2(h)}\Big)\psi\Big(x',\frac{x'}{\ep(h)}\Big)\,dx'=0.
\ee
In view of \eqref{eq:V-dr-3}, and since 
$$\intomyz {\Big(\iQ V(x',y,\xi)\,d\xi\Big):\np \varphi(z)\psi(x',y)}=0$$
by the periodicity of $\varphi$, \eqref{eq:ort-z} will be established once we show that
\be{eq:to-study}\lim_{h\to 0}\int_{\omega}\frac{(R^h)'(x')}{h}:\np \varphi\Big(\frac{x'}{\ep^2(h)}\Big)\psi\Big(x',\frac{x'}{\ep(h)}\Big)\,dx'=0.\ee

In order to prove \eqref{eq:to-study}, we adapt \cite[Lemma 3.8]{hornung.neukamm.velcic} to our framework.\\

 Since $\psi\in C^{\infty}_c(\omega;C^{\infty}_{\rm per}(Q))$ and $h\to 0$, we can assume, without loss of generality, that for $h$ small enough
$$ \dist({\rm supp }\,\psi;\partial\omega\times Q)>\Big(1+\frac{3}{\gamma_1}\Big)h.$$
 We define
 $$\ze:=\Big\{z\in\Z^2:\,Q(\ep(h)z,\ep(h))\times Q\cap{\rm supp }\,\psi\neq\emptyset\Big\},$$
 and 
 $$Q_{\ep}:=\bigcup_{z\in\ze}Q(\ep(h)z,\ep(h)).$$
 Since $0<\gamma_1<+\infty$, for $h$ small enough we have $\sqrt{2}\ep(h)<\frac{2h}{\gamma_1}$, so that
$$ \dist(Q_{\ep};\partial\omega)\geq\Big(1+\frac{3}{\gamma_1}\Big)h-\sqrt{2}\ep(h)\geq\Big(1+\frac{1}{\gamma_1}\Big)h.$$
 We subdivide
 $$\mathcal{Q}_{\ep^2}:=\Big\{Q(\ep^2(h)\lambda,\ep^2(h)):\,\lambda\in\Z^2\text{ and }Q(\ep^2(h)\lambda,\ep^2(h))\cap Q_{\ep}\neq\emptyset\Big\}$$
 into two subsets:
\begin{enumerate}
\item[(a)] ``good cubes of size $\ep^2(h)$", i.e., those which are entirely contained in a cube of size $\ep(h)$ belonging to $Q_{\ep}$, and where $(R^h)'$ is hence constant,
\item[(b)] ``bad cubes of size $\ep^2(h)$", i.e., those intersecting more than one element of $Q_{\ep}$.
\end{enumerate}
We observe that, as $\gamma_2=+\infty$,
\be{eq:last-safe-distance}\dist(\mathcal{Q}_{\ep^2};\partial\omega)\geq \dist(Q_{\ep};\partial\omega)-\sqrt{2}\ep^2(h)>h\ee
for $h$ small enough, and
\be{eq:count-zep}\# \ze\leq C\frac{|{\omega}|}{\ep^2(h)}.\ee
Moreover, if $z\in\ze$, $\lambda\in\Z^2$, and 
$$\ep^2(h)\lambda\in Q(\ep(h)z,\ep(h)-\ep^2(h)),$$ then  
$Q(\ep^2(h)\lambda,\ep^2(h))$ is a ``good cube", therefore the boundary layer of $Q(\ep(h)z,\ep(h))$, that could possibly intersect ``bad cubes" has measure given by
$$|Q(\ep(h)z,\ep(h))|-|Q(\ep(h)z,\ep(h)-\ep^2(h))|=\ep(h)^2-(\ep(h)-\ep(h)^2)^2=2\ep(h)^3-\ep(h)^4.$$
By \eqref{eq:count-zep} we conclude that the sum of all areas of ``bad cubes" intersecting $Q_{\ep}$ is bounded from above by
\be{eq:area-bad-cubes}C\frac{|{\omega}|}{\ep^2(h)}(2\ep^3(h)-\ep^4(h))\leq C\ep(h).\ee

\noindent We define the sets
\begin{align*}
\zeg:=\Big\{\lambda\in\Z^2: \exists z\in\ze\text{ s.t. }Q(\ep^2(h)\lambda,\ep^2(h))\subset Q(\ep(h)z,\ep(h))\Big\},
\end{align*}
and
$$\zeb:=\Big\{\lambda\in\Z^2:Q(\ep(h)^2\lambda,\ep^2(h))\cap Q_{\ep}\neq\emptyset\text{ and }\lambda\notin\zeg\Big\}$$
(where `g' and `b' stand for ``good'' and ``bad'', respectively).
We rewrite \eqref{eq:to-study} as 
\begin{align*}
&\int_{\omega}\frac{(R^h)'(x')}{h}:\np \varphi\Big(\frac{x'}{\ep^2(h)}\Big)\psi\Big(x',\frac{x'}{\ep(h)}\Big)\,dx'\\
&\quad=\szeg\int_{\qel}\frac{(R^h)'(x')}{h}:\np \varphi\Big(\frac{x'}{\ep^2(h)}\Big)\psi\Big(x',\frac{x'}{\ep(h)}\Big)\,dx'\\
&\qquad+\szeb\int_{\qel}\frac{(R^h)'(x')}{h}:\np \varphi\Big(\frac{x'}{\ep^2(h)}\Big)\psi\Big(x',\frac{x'}{\ep(h)}\Big)\,dx'.
\end{align*}

Since the maps $\{(R^h)'\}$ are piecewise constant on ``good cubes", by the periodicity of $\varphi$ we have
\begin{align}
\label{eq:decomposition}
&\int_{\omega}\frac{(R^h)'(x')}{h}:\np \varphi\Big(\frac{x'}{\ep^2(h)}\Big)\psi\Big(x',\frac{x'}{\ep(h)}\Big)\,dx'\\
\nonumber&\quad=\szeg\int_{\qel}\frac{(R^h)'(x')}{h}\\
\nonumber&\qquad:\np \varphi\Big(\frac{x'}{\ep^2(h)}\Big)\left(\psi\Big(x',\frac{x'}{\ep(h)}\Big)-\psi(\ep^2(h)\lambda,\ep(h)\lambda)\right)\,dx'\\
\nonumber&\qquad+\szeb\int_{\qel}\frac{(R^h)'(x')}{h}\\
\nonumber&\qquad:\np \varphi\Big(\frac{x'}{\ep^2(h)}\Big)\left(\psi\Big(x',\frac{x'}{\ep(h)}\Big)-\psi(\ep^2(h)\lambda,\ep(h)\lambda)\right)\,dx'\\
\nonumber&\qquad+\szeb\int_{\qel}\frac{(R^h)'(x')}{h}:\np \varphi\Big(\frac{x'}{\ep^2(h)}\Big)\psi(\ep^2(h)\lambda,\ep(h)\lambda)\,dx'.
\end{align}

We claim that
\be{eq:CLAIM}
\lim_{h\to 0}\Big|\szeb\int_{\qel}\frac{(R^h)'(x')}{h}:\np \varphi\Big(\frac{x'}{\ep^2(h)}\Big)\psi\Big(\ep^2(h)\lambda,\ep(h)\lambda\Big)\,dx'\Big|=0.
\ee
Indeed, by the periodicity of $\varphi$, $$\int_{\qel}\np \varphi\Big(\frac{x'}{\ep^2(h)}\Big)\,dx'=0\quad\text{for every }\lambda\in\Z^2,$$
and we have
\begin{align*}
&\Big|\szeb\int_{\qel}\frac{(R^h)'(x')}{h}:\np \varphi\Big(\frac{x'}{\ep^2(h)}\Big)\psi\Big(\ep^2(h)\lambda,\ep(h)\lambda\Big)\,dx'\Big|\\
&\quad=\Big|\szeb\int_{\qel}\frac{(R^h)'(x')-(R^h)'(\ep^2(h)\lambda)}{h}\\
&\qquad:\np \varphi\Big(\frac{x'}{\ep^2(h)}\Big)\psi\Big(\ep^2(h)\lambda,\ep(h)\lambda\Big)\,dx'\Big|.
\end{align*}
Therefore, by H\"older's inequality,
\begin{align}
\label{eq:sum-cubes}
&\Big|\szeb\int_{\qel}\frac{(R^h)'(x')}{h}:\np \varphi\Big(\frac{x'}{\ep^2(h)}\Big)\psi\Big(\ep^2(h)\lambda,\ep(h)\lambda\Big)\,dx'\Big|\\
\nonumber&\leq\frac{C}{h}\int_{\cup_{\lambda\in\zeb}\qel}|(R^h)'(x')-(R^h)'(\ep^2(h)\lambda)|\,dx'\\
\nonumber&\leq \frac{C}{h}\Big|\cup_{\lambda\in\zeb}\qel\Big|^{\tfrac12}\|(R^h)'(x')-(R^h)'(\ep^2(h)\lambda)\|_{L^2(\cup_{\lambda\in\zeb}{\qel})}.
\end{align}
Every cube $\qel$ in the previous sum intersects at most four elements of $Q_{\ep}$. For every $\lambda\in\zeb$, let $Q(\ep(h)z_i^{\lambda},\ep)$, $i=1,\cdots,4$, be such cubes, where
$$\#\{z_i^{\lambda}:i=1,\cdots,4\}\leq 4.$$
Without loss of generality, for every $\lambda\in\zeb$ we can assume that 
$$\ep^2(h)\lambda\in Q(\ep(h)z_4^{\lambda},\ep(h)),$$ so that
$$|(R^h)'(x')-(R^h)'(\ep^2(h)\lambda)|=0\quad\text{a.e. in }Q(\ep(h)z_4^{\lambda},\ep(h)).$$
Hence,
\begin{align*}
&\szeb\int_{\qel}|(R^h)'(x')-(R^h)'(\ep^2(h)\lambda)|^2\,dx'\\
&\quad=\szeb\sum_{i=1}^3\int_{\qel\cap Q(\ep(h)z_i^{\lambda},\ep(h))}|(R^h)'(x')-(R^h)'(\ep^2(h)\lambda)|^2\,dx'.
\end{align*}
Since the maps $\{R^h\}$ are piecewise constant on each set 
$$\qel\cap Q(\ep(h)z_i^{\lambda},\ep(h)),$$ there holds
$$|(R^h)'(x')-(R^h)'(\ep^2(h)\lambda)|=|(R^h)'(x')-(R^h)'(x'+\xi)|$$
for some $\xi\in\{\pm\ep^2(h)e_1,\pm\ep^2(h)e_2,\pm\ep^2(h)e_1\pm\ep^2(h)e_2\}$. 

Therefore, by \eqref{eq:last-safe-distance} and Lemma \ref{lemma:rotations1}, and since $\gamma_1\in (0,+\infty)$, we have
\begin{align}
\label{eq:bdd-osc-cubes}
&\szeb\int_{\qel}|(R^h)'(x')-(R^h)'(\ep^2(h)\lambda)|^2\,dx'\\
\nonumber&\quad\leq C\|\dist(\nh u^h; SO(3))\|_{L^2(\Omega)}^2.
\end{align}
Combining \eqref{eq:unif-bending-en}, \eqref{eq:area-bad-cubes}, \eqref{eq:sum-cubes}, and \eqref{eq:bdd-osc-cubes}, we finally get the inequality
\begin{align*}
&\Big|\szeb\int_{\qel}\frac{(R^h)'(x')}{h}:\np \varphi\Big(\frac{x'}{\ep^2(h)}\Big)\psi\Big(\ep^2(h)\lambda,\ep(h)\lambda\Big)\,dx'\Big|\\
&\quad\leq \frac{C}{h}\Big|\cup_{\lambda\in\zeb}\qel\Big|^{\tfrac12}\|\dist(\nh u^h; SO(3))\|_{L^2(\Omega)}\\
&\quad\leq C\Big|\cup_{\lambda\in\zeb}\qel\Big|^{\tfrac12}\leq C\sqrt{\ep(h)},
\end{align*}
and this concludes the proof of \eqref{eq:CLAIM}.

Estimates \eqref{eq:decomposition} and \eqref{eq:CLAIM} yield
\begin{align*}
&\lim_{h\to 0}\int_{\omega}\frac{(R^h)'(x')}{h}:\np \varphi\Big(\frac{x'}{\ep^2(h)}\Big)\psi\Big(x',\frac{x'}{\ep(h)}\Big)\,dx'\\
&\quad=\lim_{h\to 0}\sle\int_{\qel}\frac{(R^h)'(x')}{h}\\
\nonumber&\qquad:\np \varphi\Big(\frac{x'}{\ep^2(h)}\Big)\Big(\psi\Big(x',\frac{x'}{\ep(h)}\Big)-\psi(\ep^2(h)\lambda,\ep(h)\lambda)\Big)\,dx'\\
&\quad=\lim_{h\to 0}\sle\int_{\qel}\frac{(R^h)'(x')}{h}\\
\nonumber&\qquad:\np \varphi\Big(\frac{x'}{\ep^2(h)}\Big)\Big(\int_0^1\frac{d}{dt}\phi_{\ep}(\ep^2(h)\lambda+t(x'-\ep^2(h)\lambda))\,dt\Big)\,dx',
\end{align*}
where $\phi_{\ep}(x'):=\psi\Big(x',\frac{x'}{\ep(h)}\Big)$ for every $x'\in\omega$. Therefore, by the periodicity of $\varphi$
\begin{align}
\label{eq:long-estimate}
&\lim_{h\to 0}\int_{\omega}\frac{(R^h)'(x')}{h}:\np \varphi\Big(\frac{x'}{\ep^2(h)}\Big)\psi\Big(x',\frac{x'}{\ep(h)}\Big)\,dx'\\
\nonumber&\quad=\lim_{h\to 0}\Bigg[\sle \frac{\ep^2(h)}{h}\int_{\qel}(R^h)'(x'):\np \varphi\\
\nonumber&\qquad\Big(\frac{x'-\ep^2(h)\lambda}{\ep^2(h)}\Big)\Big(\int_0^1\nabla'\phi_{\ep}(\ep^2(h)\lambda+t(x'-\ep^2(h)\lambda))\cdot\frac{(x'-\ep^2(h)\lambda)}{\ep^2(h)}\,dt\Big)\,dx'\Bigg].
\end{align}
Changing coordinates in \eqref{eq:long-estimate} we get
\begin{align}
\label{eq:estimate-to-add}
&\lim_{h\to 0}\int_{\omega}\frac{(R^h)'(x')}{h}:\np \varphi\Big(\frac{x'}{\ep^2(h)}\Big)\psi\Big(x',\frac{x'}{\ep(h)}\Big)\,dx'\\
\nonumber&\quad=\lim_{h\to 0}\sle \frac{\ep^6(h)}{h}\int_{Q}(R^h)'(\ep^2(h)z+\ep^2(h)\lambda)\\
\nonumber&\qquad:\np \varphi(z)\Big(\int_0^1\nabla'\phi_{\ep}(\ep^2(h)\lambda+t\ep^2(h)z)\,dt\cdot z\Big)\,dz\\
\nonumber&\quad=\lim_{h\to 0}\left[\sle \frac{\ep^6(h)}{h}\int_{Q}(R^h)'(\ep^2(h)z+\ep^2(h)\lambda)\right.\\
\nonumber&\qquad:\np \varphi(z)\Big(\int_0^1(\nabla'\phi_{\ep}(\ep^2(h)\lambda+t\ep^2(h)z)-\nabla'\phi_{\ep}(\ep^2(h)\lambda))\,dt\cdot z\Big)\,dz\\
\nonumber&\qquad+\sle \frac{\ep^6(h)}{h}\int_{Q}(R^h)'(\ep^2(h)z+\ep^2(h)\lambda)\\
\nonumber&\qquad:\np \varphi(z)(\nabla'\phi_{\ep}(\ep^2(h)\lambda)\cdot z)\,dz\Bigg].
\end{align}

We notice that
\begin{align}
\label{eq:CLAIM2}
&\lim_{h\to 0}\Bigg[\sle \frac{\ep^6(h)}{h}\int_{Q}(R^h)'(\ep^2(h)z+\ep^2(h)\lambda)\\
\nonumber&\quad:\np \varphi(z)\Big(\int_0^1(\nabla'\phi_{\ep}(\ep^2(h)\lambda+t\ep^2(h)z)-\nabla'\phi_{\ep}(\ep^2(h)\lambda))\,dt\Big)\cdot z\,dz\Bigg]=0.
\end{align}
Indeed, since $\|(\nabla')^2\phi_{\ep}\|_{L^{\infty}(\omega\times Q;\mthree)}\leq \frac{C}{\ep^2(h)}$, we have
\begin{align*}
&\Bigg|\sle \frac{\ep^6(h)}{h}\int_{Q}(R^h)'(\ep^2(h)z+\ep^2(h)\lambda)\\
&\qquad:\np \varphi(z)\Big(\int_0^1(\nabla'\phi_{\ep}(\ep^2(h)\lambda+t\ep^2(h)z)-\nabla'\phi_{\ep}(\ep^2(h)\lambda))\,dt\cdot z\Big)\,dz\Bigg|\\
&\quad\leq C\frac{\ep^6(h)}{h}\sle\int_{Q}|(R^h)'(\ep^2(h)z+\ep^2(h)\lambda)|\|(\nabla')^2\phi_{\ep}\|_{L^{\infty}(\Omega\times Q)}|\ep^2(h)z|\,dz\\
&\quad\leq C\frac{\ep^6(h)}{h}\sle\int_{Q}|(R^h)'(\ep^2(h)z+\ep^2(h)\lambda)|\,dz\\
&\quad= C\frac{\ep^2(h)}{h}\sle\int_{\qel}|(R^h)'(x')|\,dx'\leq C\frac{\ep^2(h)}{h}\|(R^h)'\|_{L^1(\omega;\mthree)}
\end{align*}
which converges to zero by \eqref{eq:pwconst-conv} and because $\gamma_2=+\infty$. 

By \eqref{eq:CLAIM2}, estimate \eqref{eq:estimate-to-add} simplifies as
\begin{align}
\label{eq:intermediate1}
&\lim_{h\to 0}\int_{\omega}\frac{(R^h)'(x')}{h}:\np \varphi\Big(\frac{x'}{\ep^2(h)}\Big)\psi\Big(x',\frac{x'}{\ep(h)}\Big)\,dx'\\
\nonumber&\quad=\lim_{h\to 0}\sle \frac{\ep^6(h)}{h}\int_{Q}(R^h)'(\ep^2(h)z+\ep^2(h)\lambda)\\
\nonumber&\qquad:\np \varphi(z)(\nabla'\phi_{\ep}(\ep^2(h)\lambda)\cdot z)\,dz\\
\nonumber&\quad=\lim_{h\to 0}\Bigg[\sle \frac{\ep^6(h)}{h}\int_{Q}((R^h)'(\ep^2(h)z+\ep^2(h)\lambda)-(R^h)'(\ep^2(h)\lambda))\\
\nonumber&\qquad:\np \varphi(z)(\nabla'\phi_{\ep}(\ep^2(h)\lambda)\cdot z)\,dz\\
\nonumber&\quad+\sle \frac{\ep^6(h)}{h}\int_{Q}(R^h)'(\ep^2(h)\lambda):\np \varphi(z)(\nabla'\phi_{\ep}(\ep^2(h)\lambda)\cdot z)\,dz\Bigg].
\end{align}

We observe that
\begin{align}
\label{eq:CLAIM3}
&\lim_{h\to 0}\Bigg[\sle \frac{\ep^6(h)}{h}\int_{Q}((R^h)'(\ep^2(h)z+\ep^2(h)\lambda)-(R^h)'(\ep^2(h)\lambda))\\
\nonumber&\qquad:\np \varphi(z)(\nabla'\phi_{\ep}(\ep^2(h)\lambda)\cdot z)\,dz\Bigg]=0.
\end{align}
Indeed, since $\varphi\in C^{1}_{\rm per}(\R^2;\mthree)$ and $\|(\nabla')\phi_{\ep}\|_{L^{\infty}(\omega\times Q)}\leq \frac{C}{\ep(h)}$, recalling the definition of the sets $\zeb$ and $\zeg$, and applying H\"older's inequality, \eqref{eq:unif-bending-en}, \eqref{eq:area-bad-cubes}, and \eqref{eq:bdd-osc-cubes}, we obtain
\begin{align*}
&\Bigg|\sle \frac{\ep^6(h)}{h}\int_{Q}((R^h)'(\ep^2(h)z+\ep^2(h)\lambda)-(R^h)'(\ep^2(h)\lambda))\\
&\qquad:\np \varphi(z)(\nabla'\phi_{\ep}(\ep^2(h)\lambda)\cdot z)\,dz\Bigg|\\
&\quad\leq {C}\frac{\ep^5(h)}{h}\sle \int_{Q}|(R^h)'(\ep^2(h)z+\ep^2(h)\lambda)-(R^h)'(\ep^2(h)\lambda)|\,dz\\
&\quad=\frac{C\ep(h)}{h}\szeb\int_{\qel}|(R^h)'(x')-(R^h)'(\ep^2(h)\lambda)|\,dx'\\
&\quad\leq \frac{C\ep(h)}{h}\Big|\cup_{\lambda\in\zeb}\qel\Big|^{\tfrac12}\|\dist(\nh u^h;SO(3))\|_{L^2(\Omega)}\leq C\ep(h)^{\tfrac32}.
\end{align*}

Collecting \eqref{eq:intermediate1} and \eqref{eq:CLAIM3}, we deduce that
\begin{align}
\label{eq:intermediate2}
&\lim_{h\to 0}\int_{\omega}\frac{(R^h)'(x')}{h}:\np \varphi\Big(\frac{x'}{\ep^2(h)}\Big)\psi\Big(x',\frac{x'}{\ep(h)}\Big)\,dx'\\
\nonumber&\quad=\lim_{h\to 0}\Bigg[\sle \frac{\ep^6(h)}{h}\int_{Q}(R^h)'(\ep^2(h)\lambda):\np \varphi(z)(\nabla'\phi_{\ep}(\ep^2(h)\lambda)\cdot z)\,dz\Bigg].
\end{align}

Since $0<\gamma_1<+\infty$ and $\gamma_2=+\infty$, by \eqref{eq:pwconst-conv} we have
\begin{align*}
&\lim_{h\to 0}\sle \frac{\ep^6(h)}{h}\int_{Q}\fint_{\qel}(R^h)'(x'):\np \varphi(z)(\nabla'\phi_{\ep}(x')\cdot z)\,dx'\,dz\\
&\quad=\lim_{h\to 0}\frac{\ep^2(h)}{h}\int_{\omega}\iQ(R^h)'(x')\\
&\qquad:\np \varphi(z)\Big[\Big(\nabla_x\psi\Big(x',\frac{x'}{\ep(h)}\Big)+\frac{1}{\ep(h)}\nabla_y\psi\Big(x',\frac{x'}{\ep(h)}\Big)\Big)\cdot z\Big]\,dz\,dx'\\
&\quad=\frac{1}{\gamma_1}\intomyz{R'(x'):\np \varphi(z)(\nabla_y\psi(x',y)\cdot z)}=0,
\end{align*}
by the periodicity of $\psi$ with respect to $y$. We observe that if $\lambda\in\zeg$, then
\bmm
&\fint_{\qel}(R^h)'(x'):\np \varphi(z)(\nabla'\phi_{\ep}(x')\cdot z)\,dx'\\
&\quad=(R^h)'(\ep^2(h)\lambda):\fint_{\qel}\np \varphi(z)(\nabla'\phi_{\ep}(x')\cdot z)\,dx',
\end{align*}
and we obtain
\begin{align*}
&\lim_{h\to 0}\Bigg[\sle \frac{\ep^6(h)}{h}\int_{Q}(R^h)'(\ep^2(h)\lambda):\np \varphi(z)(\nabla'\phi_{\ep}(\ep^2(h)\lambda)\cdot z)\,dz\\
\nonumber&\qquad-\sle \frac{\ep^6(h)}{h}\int_{Q}\fint_{\qel}(R^h)'(x')\\
\nonumber&\qquad:\np \varphi(z)(\nabla'\phi_{\ep}(x')\cdot z)\,dx'\,dz\Bigg]\\
\nonumber&\quad=\lim_{h\to 0}\left[\szeg \frac{\ep^6(h)}{h}(R^h)'(\ep^2(h)\lambda)\right.\\
\nonumber&\qquad:\iQ\np \varphi(z)\Big[\Big(\nabla'\phi_{\ep}(\ep^2(h)\lambda)-\fint_{\qel}\nabla'\phi_{\ep}(x')\,dx'\Big)\cdot z\Big]\,dz\\
\nonumber&\qquad+\szeb \frac{\ep^6(h)}{h}\int_{Q}(R^h)'(\ep^2(h)\lambda):\np \varphi(z)(\nabla'\phi_{\ep}(\ep^2(h)\lambda)\cdot z)\,dz\\
\nonumber&\qquad-\szeb \frac{\ep^6(h)}{h}\int_{Q}\fint_{\qel}(R^h)'(x'):\np \varphi(z)(\nabla'\phi_{\ep}(x')\cdot z)\,dx'\,dz\Bigg].
\end{align*}
By the regularity of $\varphi$ and $\psi$, and the boundedness of $\{R^h\}$ in $L^{\infty}(\omega;\mthree)$,
\begin{align}
\label{eq:term1good}
&\Bigg|\szeg \frac{\ep^6(h)}{h}(R^h)'(\ep^2(h)\lambda)\\
\nonumber&\qquad:\iQ\np \varphi(z)\Bigg[\Big(\nabla'\phi_{\ep}(\ep^2(h)\lambda)-\fint_{\qel}\nabla'\phi_{\ep}(x')\,dx'\Big)\cdot z\Bigg]\,dz\Bigg|\\
\nonumber&\quad\leq C\frac{\ep^2(h)}{h}\szeg\int_{\qel}|\nabla'\phi_{\ep}(\ep^2(h)\lambda)-\nabla'\phi_{\ep}(x')|\,dx'\\
\nonumber&\quad\leq C\frac{\ep^4(h)}{h}\|\nabla^2\phi_{\ep}\|_{L^{\infty}(\omega\times Q;\mthree)}\leq C\frac{\ep^2(h)}{h},
\end{align}
which converges to zero, because $\gamma_2=+\infty$.
On the other hand,
\begin{align}
\label{eq:term1bad}
&\szeb \frac{\ep^6(h)}{h}\int_{Q}\Bigg[(R^h)'(\ep^2(h)\lambda):\np \varphi(z)(\nabla'\phi_{\ep}(\ep^2(h)\lambda)\cdot z)\,dz\\
\nonumber&\qquad-\fint_{\qel}(R^h)'(x'):\np \varphi(z)(\nabla'\phi_{\ep}(x')\cdot z)\,dx'\Bigg]\,dz\\
\nonumber&\quad=\szeb\frac{\ep^6(h)}{h}\int_{Q}(R^h)'(\ep^2(h)\lambda)\\
\nonumber&\qquad:\np \varphi(z)\Big[\Big(\nabla'\phi_{\ep}(\ep^2(h)\lambda)-\fint_{\qel}\nabla'\phi_{\ep}(x')\,dx'\Big)\cdot z\Big]\,dz\\
\nonumber&\quad+\szeb\frac{\ep^6(h)}{h}\int_{Q}\fint_{\qel}((R^h)'(\ep^2(h)\lambda)-(R^h)'(x'))\\
\nonumber&\qquad:\np \varphi(z)(\nabla'\phi_{\ep}(x')\cdot z)\,dx'\,dz.
\end{align}
Therefore, arguing as in \eqref{eq:term1good}, the first term on the right hand side of \eqref{eq:term1bad} is bounded by
$C\frac{\ep^2(h)}{h}$,
whereas by \eqref{eq:area-bad-cubes} and the boundedness of $\{R^h\}$ in $L^{\infty}(\omega;\mthree)$,
\begin{align}
\label{eq:bad2done}
&\Bigg|\szeb\frac{\ep^6(h)}{h}\int_{Q}\fint_{\qel}((R^h)'(\ep^2(h)\lambda)-(R^h)'(x'))\\
\nonumber&\qquad:\np \varphi(z)(\nabla'\phi_{\ep}(x')\cdot z)\,dx'\,dz\Bigg|\\
\nonumber&\quad\leq C\frac{\ep(h)}{h}\szeb\int_{\qel}|(R^h)'(x')-(R^h)'(\ep^2(h)\lambda)|\,dx'\\
\nonumber&\quad\leq C\frac{\ep^2(h)}{h},
\end{align}
which converges to zero as $\gamma_2=+\infty$. 

Combining \eqref{eq:intermediate2}--\eqref{eq:bad2done} we conclude that
\be{eq:rh-osc}\lim_{h\to 0}\int_{\omega}\frac{(R^h)'(x')}{h}:\np \varphi\Big(\frac{x'}{\ep^2(h)}\Big)\psi\Big(x',\frac{x'}{\ep(h)}\Big)\,dx'=0.\ee
By \eqref{eq:V-dr-3}, \eqref{eq:claim2-part1}, and \eqref{eq:rh-osc}, we obtain
$$\intomyz{\Big(V(x',y,z)-\int_Z V(x',y,\xi)\,d\xi\Big):\np \varphi(z)\psi(x',y)}=0,$$
for all $\varphi\in C^1_{\rm per}(Q;\R^3)$ and $\psi\in C^{\infty}_c(\omega;C^{\infty}_{\rm per}(Q))$. 

This completes the proof of \eqref{eq:claim-V2}.\\

\noindent\emph{Case 1, Step 2: Characterization of the limit linearized strain $G$.}\\
In order to identify the multiscale limit of the sequence of linearized strains $G^h$, by \eqref{eq:E-G}, \eqref{eq:remaining-E-gamma_1}, \eqref{eq:scaled-grad-dec}--\eqref{eq:V-dr-3} we now characterize the weak 3-scale limits of the sequences $\{x_3\nabla'\tilde{R}^he_3\}$ and $\{\frac{1}{h}(\tilde{R}^he_3-R^he_3)\}$. 

By \eqref{eq:sobolev-conv} and \cite[Theorem 1.2]{allaire.briane} there exist $S\in L^2(\omega;W^{1,2}_{\rm per}(Q;\mthree))$ and $T\in L^2(\omega\times Q; W^{1,2}_{\rm per}(Q;\mthree))$ such that
\be{eq:weak-gradRh-3scales}\nabla'\tilde{R}^h\swks \nabla'R+\nabla_yS+\nabla_zT\quad\text{weakly 3-scale},\ee
where $\iQ S(x',y)\,dy=0$ for a.e. $x'\in\omega$, and $\iQ T(x',y,z)\,dz=0$ for a.e. $x'\in\omega$, and $y\in Y$.
By \eqref{eq:unif-bending-en} and \eqref{eq:approximate-rot}, there exists $w\in L^2(\omega\times Q\times Q;\R^3)$ such that
$$\frac{1}{h}(\tilde{R}^he_3-R^he_3)\swks w\quad\text{weakly 3-scale}$$
and hence,
$$\frac{1}{h}(\tilde{R}^he_3-R^he_3)\wk w_0\quad\text{weakly in }L^2(\omega;\R^3)$$ 
where
$$w_0(x'):=\int_Q\int_Q w(x',y,z)\,dy\,dz,$$
for a.e. $x'\in\omega$. 
We claim that
\be{eq:FINALCLAIM}
\frac{1}{h}(\tilde{R}^he_3-R^he_3)\swks w_0(x')+\frac{1}{\gamma_1}S(x',y)e_3+\frac{(y\cdot\nabla')R(x')e_3}{\gamma_1},
\ee
weakly 3-scale.
We first remark that the same argument as in the proof of \eqref{eq:to-study} yields
$$\frac{R^he_3}{h}\oscz 0.$$
Moreover, since $\gamma_1\in(0,+\infty)$, by \eqref{eq:pwconst-conv}, Lemmas \ref{lemma:pw-osc} and \ref{lemma:sobolev-osc}, there holds
$$\frac{R^he_3}{h}\oscy-\frac{(y\cdot\nabla')Re_3}{\gamma_1}$$
and
$$\frac{\tilde{R}^he_3}{h}\oscy\frac{Se_3}{\gamma_1},$$
where in the latter property we used the fact that $\iQ\nabla_zT(x',y,z)\,dz=0$ for a.e. $x'\in\omega$ and $y\in Y$ by periodicity, and $\iQ S(x',y)\,dy=0$ for a.e. $x'\in\omega$.
Therefore, by Remark \ref{remark:osc}, to prove \eqref{eq:FINALCLAIM} we only need to show that
\be{eq:FINALCLAIM2}
\frac{\tilde{R}^he_3}{h}\oscz 0.
\ee

 To this purpose, fix $\varphi\in C^{\infty}_{\rm per}(Q)$, with $\int_Q \varphi(z)\,dz=0$, and $\psi\in C^{\infty}_c(\Omega; C^{\infty}_{\rm per}(Q))$, and let $g\in C^2(Q)$ be the unique periodic solution to
$$\begin{cases}
\Delta g(z)=\varphi(z)&\\
\int_Q g(z)\,dz=0.
\end{cases}$$
Set 
\be{eq:def-ge}g^{\ep}(x'):=\ep^2(h)g\Big(\frac{x'}{\ep^2(h)}\Big)\quad\text{for every }x'\in\omega,\ee so that
\be{eq:lap-ge}\Delta g^{\ep}(x')=\frac{1}{\ep^2(h)}\varphi\Big(\frac{x'}{\ep^2(h)}\Big)\quad\text{for every }x'\in\omega.\ee
By \eqref{eq:def-ge} and \eqref{eq:lap-ge}, and for $i\in\{1,2,3\}$, we obtain
\begin{align*}
&\int_{\omega}\frac{\tilde{R}^h_{i3}(x')}{h}\varphi\Big(\frac{x'}{\ep^2(h)}\Big)\psi\Big(x',\frac{x'}{\ep(h)}\Big)\,dx'\\
&\quad=\frac{\ep^2(h)}{h}\int_{\omega}\tilde{R}^h_{i3}(x')\Delta g^{\ep}(x')\psi\Big(x',\frac{x'}{\ep(h)}\Big)\,dx'.
\end{align*}
Integrating by parts, we have
\begin{align}
\label{eq:sum-many-terms}
&\int_{\omega}\frac{\tilde{R}^h_{i3}(x')}{h}\varphi\Big(\frac{x'}{\ep^2(h)}\Big)\psi\Big(x',\frac{x'}{\ep(h)}\Big)\,dx'\\
\nonumber&\quad=-\frac{\ep^2(h)}{h}\int_{\omega}\nabla'\tilde{R}^h_{i3}(x')\cdot \nabla' \Big(g^{\ep}(x')\psi\Big(x',\frac{x'}{\ep(h)}\Big)\Big)\,dx'\\
\nonumber&\qquad-\frac{\ep^2(h)}{h}\int_{\omega}\tilde{R}^h_{i3}(x')\Big(2\nabla' g^{\ep}(x')\cdot(\nabla_{x'}\psi)\Big(x',\frac{x'}{\ep(h)}\Big)\\
\nonumber&\qquad+g^{\ep}(x')(\Delta_{x'}\psi)\Big(x',\frac{x'}{\ep(h)}\Big)\Big)\,dx'\\
\nonumber&\qquad-\frac{\ep(h)}{h}\int_{\omega}\tilde{R}^h_{i3}(x')\Big[2\nabla' g^{\ep}(x')\cdot\nabla_y\psi\Big(x',\frac{x'}{\ep(h)}\Big)\\
\nonumber&\qquad+2g^{\ep}(x')({\rm div}_y\,\nabla_{x'}\psi)\Big(x',\frac{x'}{\ep(h)}\Big)\Big]\,dx'\\
\nonumber&\qquad-\frac{1}{h\ep(h)}\int_{\omega}\tilde{R}^h_{i3}(x') g^{\ep}(x')\Delta_y\psi\Big(x',\frac{x'}{\ep(h)}\Big)\,dx'.
\end{align}
Since 
$\nabla' \Big(g^{\ep}(\cdot)\psi\Big(\cdot,\frac{\cdot}{\ep(h)}\Big)\Big)\in L^{\infty}(\omega;\R^2)$,
\be{eq:term1-finalclaim}
\lim_{h\to 0}\frac{\ep^2(h)}{h}\int_{\omega}\nabla'\tilde{R}^h_{i3}(x')\cdot\nabla' \Big(g^{\ep}(x')\psi\Big(x',\frac{x'}{\ep(h)}\Big)\Big)\,dx'=0,
\ee
where we used the fact that $\gamma_2=+\infty$, and similarly,
\be{eq:term2-finalclaim}
\lim_{h\to 0}\frac{\ep^2(h)}{h}\int_{\omega}\tilde{R}^h_{i3}(x')\Big(2\nabla' g^{\ep}(x')\cdot(\nabla_{x'}\psi)\Big(x',\frac{x'}{\ep(h)}\Big)+g^{\ep}(x')(\Delta_{x'}\psi)\Big(x',\frac{x'}{\ep(h)}\Big)\Big)\,dx'=0.
\ee
Regarding the third term in the right-hand side of \eqref{eq:sum-many-terms}, we write
\begin{align}
\label{eq:term3-finalclaim}
&\frac{\ep(h)}{h}\int_{\omega}\tilde{R}^h_{i3}(x')\Big[2\nabla' g^{\ep}(x')\cdot\nabla_y\psi\Big(x',\frac{x'}{\ep(h)}\Big)+2g^{\ep}(x')({\rm div}_y\,\nabla_{x'}\psi)\Big(x',\frac{x'}{\ep(h)}\Big)\Big]\,dx'\\
\nonumber&\quad=2\frac{\ep(h)}{h}\int_{\omega}\tilde{R}^h_{i3}(x')\nabla' g\Big(\frac{x'}{\ep^2(h)}\Big)\cdot\nabla_y\psi\Big(x',\frac{x'}{\ep(h)}\Big)\,dx'\\
\nonumber&\qquad+\frac{2\ep^3(h)}{h}\int_{\omega}\tilde{R}^h_{i3}(x') g\Big(\frac{x'}{\ep^2(h)}\Big)({\rm div}_y\,\nabla_{x'}\psi)\Big(x',\frac{x'}{\ep(h)}\Big)\,dx'.
\end{align}
By the regularity of $g$ and $\psi$,
$$\nabla' g\Big(\frac{x'}{\ep^2(h)}\Big)\cdot\nabla_y\psi\Big(x',\frac{x'}{\ep(h)}\Big)\sstts \nabla g(z)\nabla_y\psi(x',y)\quad\text{strongly 3-scale}.$$
Therefore, by \eqref{eq:sobolev-conv}, and since $0<\gamma_1<+\infty$ and $\gamma_2=+\infty$, we obtain
\begin{align}
\label{eq:term3-finalclaim3}
&\lim_{h\to 0}\left[\frac{\ep(h)}{h}\int_{\omega}\tilde{R}^h_{i3}(x')\Big[2\nabla' g^{\ep}(x')\cdot\nabla_y\psi\Big(x',\frac{x'}{\ep(h)}\Big)\right.\\
\nonumber&\qquad\left.+2g^{\ep}(x')({\rm div}_y\,\nabla_{x'}\psi)\Big(x',\frac{x'}{\ep(h)}\Big)\Big]\,dx'\right]\\
\nonumber&\quad=\frac{2}{\gamma_1}\intomyz{{R}_{i 3}(x')\nabla g(z)\cdot\nabla_y\psi(x',y)}=0,
\end{align}
where the last equality is due to the periodicity of $\psi$ in the $y$ variable. 

Again by the regularity of $g$ and $\psi$,
$$g\Big(\frac{x'}{\ep^2(h)}\Big)\Delta_y\psi\Big(x',\frac{x'}{\ep(h)}\Big)\sstts g(z)\Delta_y\psi(x',y)\quad\text{strongly 3-scale},$$
hence, by \eqref{eq:sobolev-conv}, and since $0<\gamma_1<+\infty$ and $\psi\in C^{\infty}_c(\omega; C^{\infty}_{\rm per}(Q))$, the fourth term in the right-hand side of \eqref{eq:sum-many-terms} satisfies
\begin{align}
\label{eq:term4-finalclaim}
&\lim_{h\to 0}\frac{1}{h\ep(h)}\int_{\omega}\tilde{R}^h_{i3}(x') g^{\ep}(x')\Delta_y\psi\Big(x',\frac{x'}{\ep(h)}\Big)\Big)\,dx'\\
\nonumber&\quad=\frac{1}{\gamma_1}\intomyz {R_{i 3}(x')g(z)\Delta_y\psi(x',y)}=0.
\end{align}
Claim \eqref{eq:FINALCLAIM2}, and thus \eqref{eq:FINALCLAIM}, follow now by combining \eqref{eq:sum-many-terms} with \eqref{eq:term1-finalclaim}--\eqref{eq:term4-finalclaim}.\\

\noindent\emph{Case 1, Step 3: Characterization of E.}\\
By \eqref{eq:pwconst-conv}, and by collecting \eqref{eq:scaled-grad-dec}, \eqref{eq:rh-lim}, \eqref{eq:V-dr-3}, \eqref{eq:weak-gradRh-3scales}, and \eqref{eq:FINALCLAIM}, we deduce the characterization
\begin{align*}
R(x')G(x,y,z)&=\Big(\nabla'r(x')+\nabla_y\hat{\phi}_1(x,y)+\nabla_z\hat{\phi}_2(x,y,z)\Big|\frac{1}{\gamma_1}\partial_{x_3}\hat{\phi}_1(x,y)\Big)\\
\nonumber&\quad+\left(V(x',y,z)\Big|w_0(x')+\frac{1}{\gamma_1}S(x',y)e_3+\frac{(y\cdot \nabla')R'(x')}{\gamma_1}e_3\right)\\
\nonumber&\quad+x_3\Big(\nabla'R(x')e_3+\nabla_yS(x',y)e_3+\nabla_zT(x',y,z)e_3|0\Big)
\end{align*}
for a.e. $x\in\Omega$ and $y,z\in Q$, where $r\in W^{1,2}(\omega;\R^3)$, $\hat{\phi}_1\in L^2(\omega;W^{1,2}((-\tfrac12,\tfrac12);\newline
W^{1,2}_{\rm per}(Q;\R^3)))$, $w_0\in L^2(\omega;\R^3)$, $S\in L^2(\omega;W^{1,2}_{\rm per}(Q;\mthree))$, $V\in L^2(\omega\times Q\times Q;\M^{3\times 2})$, $\hat{\phi}_2\in L^2(\Omega\times Q; W^{1,2}_{\rm per}(Q;\R^3))$, and $T\in L^2(\omega\times Q;W^{1,2}_{\rm per}(Q;\mthree))$. Therefore, by \eqref{eq:claim-V2}
\begin{align*}
&{\rm sym}\,G(x,y,z)-\iQ{\rm sym}\,G(x,y,\xi)\,d\xi\\
\nonumber&\quad={\rm sym}\,\Bigg[R(x')^T\Big(V(x',y,z)-\iQ V(x',y,z)+\nabla_z\hat{\phi}_2(x,y,z)\Big|0\Big)\\
\nonumber&\qquad+x_3R(x')^T\Big(\nabla_zT(x',y,z)e_3\Big|0\Big)\Bigg]\\
\nonumber&\quad={\rm sym}\,\Bigg[R(x')^T\Big(\nabla_zv(x',y,z)+\nabla_z\hat{\phi}_2(x,y,z)+x_3\nabla_z T(x',y,z)e_3\Big|0\Big)\Bigg],
\end{align*}
where $Te_3, \tilde{\tilde{v}}\in L^2(\omega\times Q; W^{1,2}_{\rm per}(Q;\R^3))$.
The thesis follows now by \eqref{eq:E-G}, \eqref{eq:first-part-E-g1}, and by setting
$${\phi}_2:=R^T(v+\hat{\phi}_2+x_3Te_3).$$
for a.e. $x\in\Omega$, and $y,z\in Q$.\\

\medskip

\noindent\textbf{Case 2: $\gamma_1=+\infty$ and $\gamma_2=+\infty$}.\\
The proof is very similar to the first case where $0<\gamma_1<+\infty$. We only outline the main modifications.

Arguing as in \cite[Proof of Proposition 3.2]{hornung.neukamm.velcic}, in order to construct the sequence $\{R^h\}$, we apply Lemma \ref{lemma:rotations1} with
$$\delta(h):=\Big(2\ceil[\Big]{\frac{h}{\ep(h)}}+1\Big)\ep(h).$$
This way, 
$$\lim_{h\to 0}\frac{h}{\delta(h)}= \frac12,$$
and the maps $R^h$ are piecewise constant on cubes of the form $Q(\delta(h)z,\delta(h))$, with $z\in\Z^2$. In particular, since $\Big\{\frac{\delta(h)}{\ep(h)}\Big\}$ is a sequence of odd integers, by Lemma \ref{lemma:integers-appendix} the maps $R^h$ are piecewise constant on cubes of the form $Q(\ep(h)z,\ep(h))$ with $z\in\Z^2$, and \eqref{eq:approximate-rot} holds true. Defining $\{r^h\}$ as in \eqref{eq:def-rh}, we obtain equality \eqref{eq:scaled-grad-dec}. By Theorem \ref{thm:mod3scales-limits}(i), there exist $r\in W^{1,2}(\omega;\R^3)$, $\hat{\phi}_1\in L^2(\Omega;W^{1,2}_{\rm per}(Q;\R^3))$, $\hat{\phi}_2\in L^2(\Omega\times Q;W^{1,2}_{\rm per}(Q;\R^3))$, and \mbox{$\bar{\phi}\in L^2\big(\omega;W^{1,2}\big(\big(-\tfrac12,\tfrac12\big);\R^3\big)\big)$} such that
\be{eq:rhinfty}\nh r^h\wks (\nabla' r+\nabla_y \hat{\phi}_1+\nabla_z\hat{\phi}_2|\partial_{x_3}\bar{\phi})\quad\text{weakly dr-3-scale}.\ee
Moreover, \eqref{eq:first-part-E-g1} now becomes
\begin{align*}
&{\rm sym}\,\iQ G(x,y,\xi)\,d\xi\\
\nonumber&\quad=\Big(\begin{array}{cc}x_3\Pi^u(x')+{\rm sym }\,B(x')&0\\
0&0\end{array}\Big)+{\rm sym }\Big(\nabla_y \phi_1(x,y)\Big|\partial_{x_3}\bar{\phi}\Big)
\end{align*}
for a.e. $x\in\Omega$ and $y\in Y$, where $B\in L^2(\omega;\M^{2\times 2})$. Arguing as in Step 1--Step 3 of Case 1, we obtain the characterization
\begin{align*}
E(x,y,z)&=\Bigg(\begin{array}{ccc}x_3\Pi^u(x')+{\rm sym\,}B(x')&0\\
0&0\end{array}\Bigg)\\
&\quad+{\rm sym } (\nabla_y {\phi}_1(x,y)|d(x))+{\rm sym }(\nabla_z{\phi}_2(x,y,z)|0),
\end{align*}
with $d:=\partial_{x_3}\bar{\phi}\in L^2(\Omega;\R^3)$, $\phi_1\in L^2(\Omega;W^{1,2}_{\rm per}(Q;\R^3))$, and $\phi_2\in L^2(\Omega\times Q;\newline W^{1,2}_{\rm per}(Q;\R^3))$.\\

\noindent\textbf{Case 3: $\gamma_1=0${ and }$\gamma_2=+\infty$}.\\
The structure of the proof is similar to that of Cases 1 and 2, therefore we only outline the main steps and key points, leaving the details to the reader.

We first apply Lemma \ref{lemma:rotations1} with 
$$\delta(h):=\Big(2\ceil[\Big]{\frac{h}{\ep^2(h)}}+1\Big)\ep^2(h),$$
 and by Lemma \ref{lemma:integers-appendix} we construct 
 $$\{R^h\}\subset L^{\infty}(\omega; SO(3))\quad\text{and}\quad\{\tilde{R}^h\}\subset W^{1,2}(\omega;\mthree),$$ satisfying \eqref{eq:approximate-rot}, and with $R^h$ piecewise constant on every cube of the form 
 $$Q(\ep^2(h)z,\ep^2(h)),\quad\text{with }z\in\Z^2.$$
 
  Arguing as in Case 1, we obtain the convergence properties in \eqref{eq:pwconst-conv} and \eqref{eq:sobolev-conv}, and the identification of $E$ reduces to establishing a characterization of the weak 3-scale limit $G$ of the sequence $\{G^h\}$ defined in \eqref{eq:Gh}. In view of \cite[Proposition 3.2]{velcic}, there exist $B\in L^2(\omega;\M^{2\times 2})$, $\xi\in L^2(\Omega; W^{1,2}_{\rm per}(Q;\R^2))$, $\eta\in L^2(\omega;W^{2,2}_{\rm per}(Q;\R^2))$, and $g_i\in L^2(\Omega\times Y)$, $i=1,2,3,$ such that
\bm{eq:first-part-E-bis}
\iQ E(x,y,z)\,dz&={\rm sym}\,\iQ G(x,y,z)\,dz\\
\nonumber&=\Big(\begin{array}{cc}x_3\Pi^u(x')+{\rm sym }\,B(x')&0\\
0&0\end{array}\Big)\\
\nonumber&\qquad+\Bigg(\begin{array}{cc}{\rm sym }\nabla_y\xi(x,y)+x_3\nabla_y^2\eta(x',y)&g_1(x,y)\\
\phantom{g}&g_2(x,y)\\
g_1(x,y)\quad g_2(x,y)&g_3(x,y)\end{array}\Bigg)
\end{align}
for a.e. $x\in\Omega$ and $y\in Y$. We consider the maps $\{\bar{u}^h\}$ and $\{r^h\}$ defined in \eqref{eq:def-uh} and \eqref{eq:def-rh}, and we perform the decomposition
in \eqref{eq:scaled-grad-dec}. By Theorem \ref{thm:mod3scales-limits} (iii) there exist maps $r\in W^{1,2}(\omega;\R^3)$, $\hat{\phi}_1\in L^2(\omega;W^{1,2}_{\rm per}(Q;\R^3))$, $\hat{\phi}_2\in L^2(\Omega\times Q;W^{1,2}_{\rm per}(Q;\R^3))$, and $\bar{\phi}\in L^2\big(\omega\times Q; W^{1,2}\big(\big(-\tfrac12,\tfrac12\big);\R^3\big)\big)$ such that
$$\nh r^h\wks (\nabla' r+\nabla_y \hat{\phi}_1+\nabla_z\hat{\phi}_2|\partial_{x_3}\bar{\phi})\quad\text{weakly dr-3-scale}.$$
Defining $V$ as in \eqref{eq:V-dr-3}, we first need to show that
\be{eq:claimV-gamma2-bis}V(x',y,z)-\iQ V(x',y,z)\,dz=\nabla_zv(x',y,z)\ee
for a.e. $x'\in\omega$, and $y,z\in Q$, for some $v\in L^2(\omega\times Q;W^{1,2}_{\rm per}(Q;\R^3))$. 

As in Case 1--Step 1, by \cite[Lemma 3.7]{allaire.briane} and by a density argument, to prove \eqref{eq:claimV-gamma2-bis} it is enough to show that
\be{eq:ort-V-gamma2-bis}
\intomyz{\big(V(x',y,z)-\iQ V(x',y,z)\,dz\big):\np\varphi(z)\phi(y)\psi(x')}=0\ee
for every $\varphi\in C^{\infty}_{\rm per}(Q;\R^3)$, $\phi\in C^{\infty}_{\rm per}(Q)$ and $\psi\in C^{\infty}_c(\omega)$. 

Fix $\varphi\in C^{\infty}_{\rm per}(Q;\R^3)$, $\phi\in C^{\infty}_{\rm per}(Q)$, $\psi\in C^{\infty}_c(\omega)$, and set 
$$\varphi^{\ep}(x'):=\ep^2(h)\phi\Big(\frac{x'}{\ep(h)}\Big)\varphi\Big(\frac{x'}{\ep^2(h)}\Big)\quad\text{for every }x'\in\R^2.$$
Integrating by parts and applying Riemann-Lebesgue lemma (see \cite{fonseca.leoni}) we deduce 
\bm{eq:case3-additional1-bis}
&\lim_{h\to 0}\int_{\omega}{\frac{\nh \bar{u}^h-(R^h)'}{h}:\np \varphi^{\ep}(x')\psi(x')}\\
\nonumber&\quad=\intomyz{V(x',y,z):\npp\varphi(z)\phi(y)\psi(x')}.
\end{align}
In view of \eqref{eq:case3-additional1-bis}, \eqref{eq:ort-V-gamma2-bis} reduces to showing that
\be{eq:claim-gamma2-bis}\lim_{h\to 0}\int_{\omega}{\frac{(R^h)'(x')}{h}:\np \varphi^{\ep}(x')\psi(x')\,dx'}=0.\ee
The key idea to prove \eqref{eq:claim-gamma2-bis} is to work on cubes $Q(\ep^2(h)z,\ep^2(h))$, with $z\in\Z^2$. Exploiting the periodicity of $\varphi$ and the fact that $\{R^h\}$ is piecewise constant on such cubes, we add and subtract the values of $\phi$ and $\psi$ in $\ep^2(h)z$, and use the smoothness of the maps to control their oscillations on each cube $Q(\ep^2(h)z,\ep^2(h))$, for $z\in\Z^2$.
Defining
$$\hat{\Z}^{\ep}:=\{z\in\Z^2:\,Q(\ep^2(h)z,\ep^2(h))\cap\, {\rm supp }\,\psi\neq\emptyset\},$$
a crucial point is to prove {the equivalent of \eqref{eq:intermediate2}, that is to show that} 
\bm{eq:C2-goes-zero-bis}
&\lim_{h\to 0}\frac{\ep^5(h)}{h}\slet\fint_{\qel}\iQ\Big\{(R^h)'(x')\\
\nonumber&\qquad:\np\varphi(z)\Big[\nabla'\phi\Big(\frac{x'}{\ep(h)}\Big)\cdot z\Big]\psi(x')\Big\}\,dz\,dx'=0.
\end{align}
This is achieved by adding and subtracting in \eqref{eq:C2-goes-zero-bis} the function $\frac{\tilde{R}^h}{h}$, i.e.,
\bmm
&\frac{\ep^5(h)}{h}\slet\fint_{\qel}\iQ(R^h)'(x'):\np\varphi(z)\Big[\nabla'\phi\Big(\frac{x'}{\ep(h)}\Big)\cdot z\Big]\psi(x')\,dz\,dx'\\
\nonumber&\quad=\frac{\ep^5(h)}{h}\slet\iQ\fint_{\qel}\Big\{\big((R^h)'(x')-(\tilde{R}^h)'(x')\big)\\
\nonumber&\qquad:\np\varphi(z)\Big[\nabla'\phi\Big(\frac{x'}{\ep(h)}\Big)\cdot z\Big]\psi(x')\Big\}\,dz\,dx'\\
\nonumber&\qquad+\frac{\ep^5(h)}{h}\slet\fint_{\qel}\iQ(\tilde{R}^h)'(x'):\np\varphi(z)\Big[\nabla'\phi\Big(\frac{x'}{\ep(h)}\Big)\cdot z\Big]\psi(x')\,dz\,dx'.
\end{align*}

By \eqref{eq:unif-bending-en}, \eqref{eq:approximate-rot} and by the regularity of the test functions $\phi, \varphi$, and $\psi$, we have
\bm{eq:464star}
&\Bigg|\frac{\ep^5(h)}{h}\slet\fint_{\qel}\iQ\Big\{\big((R^h)'(x')-(\tilde{R}^h)'(x')\big)\\
\nonumber&\qquad:\np\varphi(z)\Big[\nabla'\phi\Big(\frac{x'}{\ep(h)}\Big)\cdot z\Big]\psi(x')\,dz\Big\}\,dx'\Bigg|\\
\nonumber&\qquad\leq C\ep(h) \Big\|\frac{(R^h)'-(\tilde{R}^h)'}{h}\Big\|_{L^2(\omega;\M^{3\times 2})}\leq C\ep(h).
\end{align}
Finally, by \eqref{eq:sobolev-conv} and \cite[Theorem 1.2]{allaire.briane}, there exist $S\in L^2(\omega;W^{1,2}_{\rm per}(Q;\mthree))$, and $T\in L^2(\omega\times Q;W^{1,2}_{\rm per}(Q;\mthree))$ such that
\be{eq:sob-grad-gamma2-bis}
\nabla'\tilde{R}^h\swks\nabla'R+\nabla_y S+\nabla_z T\quad\text{weakly 3-scale},
\ee
where $\iQ S(x',y)\,dy=0$ for a.e. $x'\in\omega$, and $\iQ T(x',y,z)\,dz=0$ for a.e. $x'\in\omega$ and $y\in Q$.
By Lemma \ref{lemma:sobolev-osc}, 
$$\frac{\tilde{R}^h}{\ep(h)}\oscy S,$$
and hence
\be{eq:star-star-gamma2-bis}\lim_{h\to 0}\int_{\omega}\frac{(\tilde{R}^h)'(x')}{\ep(h)}\nabla'\phi\Big(\frac{x'}{\ep(h)}\Big)\psi(x')\,dx'
=\int_{\omega}\int_Q S'(x',y)\nabla'\phi(y)\psi(x')\,dx'\,dy.\ee
Since $\gamma_2=+\infty$, \eqref{eq:star-star-gamma2-bis} yields
\bmm
&\lim_{h\to 0}\frac{\ep^5(h)}{h}\slet\fint_{\qel}\iQ\Big\{(\tilde{R}^h)'(x')\\
\nonumber&\qquad:\np\varphi(z)\Big[\nabla'\phi\Big(\frac{x'}{\ep(h)}\Big)\cdot z\Big]\psi(x')\Big\}\,dz\,dx'\\
&\quad=\lim_{h\to 0}\frac{\ep^2(h)}{h}\int_{\omega}\iQ\Big\{\frac{(\tilde{R}^h)'(x')}{\ep(h)}:\np\varphi(z)\Big[\nabla'\phi\Big(\frac{x'}{\ep(h)}\Big)\cdot z\Big]\psi(x')\Big\}\,dz\,dx\\
&\quad=\frac{1}{\gamma_2}\iOQ S'(x',y):\np\varphi(z)[\nabla'\phi(y)\cdot z]\psi(x')\,dz\,dx'=0
\end{align*}
which, together with \eqref{eq:464star}, implies \eqref{eq:C2-goes-zero-bis}.

Once the proof of \eqref{eq:claimV-gamma2-bis} is completed, to identify $E$ we need to characterize the weak 3-scale limit of the scaled linearized strains $G^h$ (see \eqref{eq:Gh}, \eqref{eq:gh-3s} and \eqref{eq:E-G}). By \eqref{eq:scaled-grad-dec} this reduces to study the weak 3-scale limit of the sequence 
$$\Big\{\frac{R^he_3-\tilde{R}^he_3}{h}\Big\}.$$ By \eqref{eq:unif-bending-en} and \eqref{eq:approximate-rot}, there exists $w\in L^2(\omega\times Q\times Q;\R^3)$ such that
$$\frac{(\tilde{R}^h-R^h)}{h}\swks w(x',y,z)\quad\text{weakly 3-scale}.$$
We claim that
\be{eq:claim-osc-gamma2-bis}w(x',y,z)-\iQ w(x',y,z)\,dz=0\ee
for a.e. $x'\in\omega$, and $y,z\in Q$. To prove \eqref{eq:claim-osc-gamma2-bis}, by Remark \ref{remark:osc}, we have to show that
$$\frac{\tilde{R}^he_3-R^he_3}{h}\oscz 0.$$
A direct application of the argument in the proof of \eqref{eq:claim-gamma2-bis} yields
$$\frac{{R}^he_3}{h}\oscz 0,$$
therefore \eqref{eq:claim-osc-gamma2-bis} is equivalent to proving that
$$\frac{\tilde{R}^he_3}{h}\oscz 0$$
which follows arguing similarly to Case 1-Step 2, proof of \eqref{eq:FINALCLAIM2}.

Finally, with an argument similar to that of Case 1, Step 3, and combining  \eqref{eq:claimV-gamma2-bis} with \eqref{eq:sob-grad-gamma2-bis}, and \eqref{eq:claim-osc-gamma2-bis}, we obtain
\bmm
&R(x')G(x,y,z)-\iQ R(x')G(x,y,z)\,dz\\
&\quad=\big(\nabla_zv(x',y,z)+\nabla_z\hat{\phi}_2(x,y,z)+x_3\nabla_z T(x',y,z)e_3|0\big)
\end{align*}
for a.e. $x\in\Omega$, and $y,z\in Q$, where $v, Te_3\in L^2(\omega\times Q;W^{1,2}_{\rm per}(Q;\R^3))$, and $\hat{\phi}_2\in L^2(\Omega\times Q;W^{1,2}_{\rm per}(Q;\R^3))$.

By \eqref{eq:E-G},
$$E(x,y,z)-\iQ E(x,y,z)\,dz={\rm sym }\,(\nabla_z\phi(x,y,z)|0)$$
for a.e. $x\in\Omega,\text{ and }y,z\in Q$,
where $\phi:=R^T(v+\hat{\phi}_2+x_3Te_3)$. In view of \eqref{eq:first-part-E-bis} we conclude that
\bmm
&E(x,y,z)=\Big(\begin{array}{cc}x_3\Pi^u(x')+{\rm sym }\,B(x')&0\\
0&0\end{array}\Big)\\
&\quad+\Bigg(\begin{array}{cc}{\rm sym }\,\nabla_y\xi(x,y)+x_3\nabla_y^2\eta(x',y)&g_1(x,y)\\
\phantom{g}&g_2(x,y)\\
g_1(x,y)\quad g_2(x,y)&g_3(x,y)\end{array}\Bigg)+{\rm sym }\,(\nabla_z\phi(x,y,z)|0)
\end{align*}
for a.e. $x\in\Omega$, and $y,z\in Q$, where $B\in L^2(\omega;\M^{2\times 2})$, $\xi\in L^2(\Omega; W^{1,2}_{\rm per}(Q;\R^2))$, $\eta\in L^2(\omega;W^{2,2}_{\rm per}(Q))$, $g_i\in L^2(\Omega\times Y)$, $i=1,2,3,$, and ${\phi}\in L^2(\Omega\times Q;W^{1,2}_{\rm per}(Q;\R^3))$. The thesis follows now by \eqref{eq:c0inf}.
\end{proof}
\section{The $\Gamma$-liminf inequality}
With the identification of the limit linearized stress obtained in Section \ref{section:E}, we now find a lower bound for the effective limit energy associated to sequences of deformations with uniformly small three-dimensional elastic energies, satisfying \eqref{eq:intro-liminf}.
\label{section:liminf}
\begin{theorem}
\label{thm:liminf}
Let $\gamma_1\in [0,+\infty]$ and let $\gamma_2=+\infty$. Let $\{u^h\}\subset W^{1,2}(\Omega;\R^3)$ be a sequence of deformations satisfying the uniform energy estimate \eqref{eq:uniform-en-estimate} and converging to $u\in W^{2,2}(\omega;\R^3)$ as in Theorem \ref{thm:compactness-def}. Then, 
$$\liminf_{h\to 0}\frac{\mathcal{E}^h(u^h)}{h^2}\geq \frac{1}{12}\int_{\omega}\overline{\mathscr{Q}}^{\gamma_1}_{\rm hom}(\Pi^u(x'))\,dx',$$
where $\Pi^u$ is the map defined in \eqref{eq:def-Pi}, and
\begin{enumerate}
\item[(a)] if $\gamma_1=0$, for every $A\in\M^{2\times 2}_{\rm sym}$
\begin{align}
\label{eq:def-q-hom-bar-0-infty}
\overline{\mathscr{Q}}^{0}_{\rm hom}(A)&:=\inf\Bigg\{\int_{\big(-\tfrac12,\tfrac12\big)\times Q}{\mathscr{Q}}_{\rm hom}\Bigg(y,\Big(\begin{array}{cc}x_3A+B&0\\0&0\end{array}\Big)\\
\nonumber &+{\rm sym\,}\left.\left(\begin{array}{cc}{\rm sym }\,\nabla_y\xi(x_3,y)+x_3\nabla_y^2\eta(y)&g_1(x_3,y)\\
\phantom{g}&g_2(x_3,y)\\
g_1(x_3,y)\quad g_2(x_3,y)&g_3(x_3,y)\end{array}\right)\right):\\
\nonumber &\xi\in L^2\big(\iv; W^{1,2}_{\rm per}(Q;\R^2)\big),\, \eta\in W^{2,2}_{\rm per}(Q), \\
\nonumber &g_i\in L^2\big(\iv\times Q), i=1,2,3,\,B\in\M^{2\times 2}_{\rm sym}\Bigg\};
\end{align}
\item[(b)] if $0<\gamma_1<+\infty$, for every $A\in\M^{2\times 2}_{\rm sym}$
\begin{align}
\label{eq:def-q-hom-bar-gamma1-infty}
\overline{\mathscr{Q}}^{\gamma_1}_{\rm hom}(A)&:=\inf\Bigg\{\int_{\big(-\tfrac12,\tfrac12\big)\times Q}{\mathscr{Q}}_{\rm hom}\Big(y,\Big(\begin{array}{cc}x_3A+B&0\\0&0\end{array}\Big)\\
\nonumber &+{\rm sym\,}\left(\nabla_y \phi_1(x_3,y)\Big|\frac{\partial_{x_3}\phi_1(x_3,y)}{\gamma_1}\right)\Bigg):\\
\nonumber &\phi_1\in W^{1,2}\left((-\tfrac12,\tfrac12);W^{1,2}_{\rm per}(Q;\R^3)\right),\,B\in\M^{2\times 2}_{\rm sym}\Bigg\};
\end{align}
  \item[(c)] if $\gamma_1=+\infty$, for every $A\in\M^{2\times 2}_{\rm sym}$ 
\begin{align}
\label{eq:def-q-hom-bar-infty-infty}
\overline{\mathscr{Q}}^{\infty}_{\rm hom}(A)&:=\inf\Bigg\{\int_{\big(-\tfrac12,\tfrac12\big)\times Q}{\mathscr{Q}}_{\rm hom}\Big(y,\Big(\begin{array}{cc}x_3A+B&0\\0&0\end{array}\Big)\\
\nonumber &+{\rm sym\,}(\nabla_y \phi_1(x_3,y)|d(x_3))\Bigg):\, d\in L^2((-\tfrac12,\tfrac12);\R^3),\\
\nonumber &\quad \phi_1\in L^{2}((-\tfrac12,\tfrac12);W^{1,2}_{\rm per}(Q;\R^3)),\text{ and }B\in\M^{2\times 2}_{\rm sym}\Bigg\}
\end{align}
 where
\be{eq:def-q-hom-gamma1-infty}
{\mathscr{Q}}_{\rm hom}(y,C):=\inf\Big\{\int_{Q}{\mathscr{Q}}\big(y,z, C+{\rm sym\,}\big(\nabla \phi_2(z)\big|0\big)\big):
\phi_2\in W^{1,2}_{\rm per}(Q;\R^3)\Big\}\ee
for a.e. $y\in Q$, and for every $C\in\mthree_{\rm sym}$. 
\end{enumerate}
\end{theorem}
\begin{proof}
The proof is an adaptation of \cite[Proof of Theorem 2.4]{hornung.neukamm.velcic}. For the convenience of the reader, we briefly sketch it in the case $0<\gamma_1<+\infty$. The proof in the cases $\gamma_1=+\infty$ and $\gamma_1=0$ is analogous.

Without loss of generality, we can assume that $\fint_{\Omega}u^h(x)\,dx=0$. By assumption $(H2)$ and by Theorem \ref{thm:compactness-def}, $u\in W^{2,2}(\omega;\R^3)$ is an isometry, with 
$$u^h\to u\quad\text{strongly in }L^2(\Omega;\R^3)$$
and
$$\nh u^h\to (\nabla' u|n_u)\quad\text{strongly in }L^2(\Omega;\mthree),$$
where the vector $n_u$ is defined according to \eqref{eq:normal1} and \eqref{eq:normal2}. By Theorem \ref{thm:limit stresses} there exists $E\in L^2(\Omega\times Q\times Q;\mthree)$ such that, up to the extraction of a (not relabeled) subsequence, 
$$E^h:=\frac{\sqrt{(\nh u^h)^T\nh u^h}-Id}{h}\wks E\quad\text{weakly dr-3-scale},$$
with
\begin{align}
\label{eq:structure-limit-E}E(x,y,z)&=\Big(\begin{array}{cc}{\rm sym\,}B(x')+x_3\Pi^u(x')&0\\0&0\end{array}\Big)\\
\nonumber &+{\rm sym\,}\Big(\nabla_y \phi_1(x,y)\Big|\frac{\partial_{x_3}\phi_1(x,y)}{\gamma_1}\Big)+{\rm sym\,}\big(\nabla_z \phi_2(x,y,z)\big|0\big),
\end{align}
for a.e. $x'\in\omega$, and $y,z\in Q$, where
$B\in L^2(\omega;\M^{2\times 2})$, $\phi_1\in L^{2}(\omega; W^{1,2}((-\tfrac12,\tfrac12);$ $W^{1,2}_{\rm per}(Q;\R^3))$, and $\phi_2\in L^{2}(\omega\times Q;W^{1,2}_{\rm per}(Q;\R^3))$.
Arguing as in \cite[Proof of Theorem 6.1 (i)]{friesecke.james.muller}, by performing a Taylor expansion around the identity, and by Lemma \ref{lemma:mod3scales-lsc} we deduce that\begin{align*}
\liminf_{h\to 0}\frac{\mathcal{E}^h(u^h)}{h^2}&\geq\liminf_{h\to 0}\int_{\Omega}{\mathscr{Q}}\Big(\frac{x'}{\ep(h)},\frac{x'}{\ep^2(h)}, E^h(x)\Big)\,dx\\
&\geq\intoyz{\mathscr{Q}(y,z,E(x,y,z))}.
\end{align*}
By \eqref{eq:def-q-hom-bar-gamma1-infty}, \eqref{eq:def-q-hom-gamma1-infty}, and \eqref{eq:structure-limit-E}, we finally conclude that
\begin{align*}
\liminf_{h\to 0}\frac{\mathcal{E}^h(u^h)}{h^2}&\geq \intoy{{\mathscr{Q}}_{\rm hom}\Bigg(y,\left(\begin{array}{cc}{\rm sym\,}B(x')+x_3\Pi^u(x')&0\\0&0\end{array}\right)\\
&\quad+{\rm sym\,}\left(\nabla_y \phi_1(x,y)\Big|\frac{\partial_{x_3}\phi_1(x,y)}{\gamma_1}\right)\Bigg)}\\
&\geq \int_{\Omega}\overline{\mathscr{Q}}_{\rm hom}^{\gamma_1}(x_3\Pi^u(x'))\,dx=\int_{\Omega}x_3^2\overline{\mathscr{Q}}_{\rm hom}^{\gamma_1}(\Pi^u(x'))\,dx\\
&=\tfrac{1}{12}\int_{\omega}\overline{\mathscr{Q}}_{\rm hom}^{\gamma_1}(\Pi^u(x'))\,dx'.
\end{align*}
\end{proof}
\section{The $\Gamma$-limsup inequality: construction of the recovery sequence}
\label{section:limsup}
Let $W^{2,2}_{R}(\omega;\R^3)$ be the set of all $u\in W^{2,2}(\omega;\R^3)$ satifying \eqref{eq:normal1}. 
Let $\mathcal{A}(\omega)$ be the set of all $u\in W^{2,2}_{R}(\omega;\R^3)\cap C^{\infty}(\bar{\omega};\R^3)$ such that, for all $B\in C^{\infty}(\bar{\omega};\M^{2\times 2}_{\rm sym})$ with $B=0$ in a neighborhood of 
$$\{x'\in\omega:\Pi^u(x')=0\}$$
(where $\Pi^u$ is the map defined in \eqref{eq:def-Pi}), there exist $\alpha\in C^{\infty}(\bar{\omega})$ and $g\in C^{\infty}(\bar{\omega};\R^2)$ such that
\be{eq:star-limsup}B={\rm sym}\,\nabla' g+\alpha\Pi^u.\ee
\begin{remark}
Note that for $u\in W^{2,2}_{R}(\omega;\R^3)\cap C^{\infty}(\bar{\omega};\R^3)$, condition \eqref{eq:star-limsup} (see \cite[Lemmas 4.3 and 4.4]{hornung.neukamm.velcic}), is equivalent to writing 
\be{eq:star-star-limsup}B={\rm sym }\,((\nabla' u)^T\nabla'V)\ee
for some $V\in C^{\infty}(\bar{\omega};\R^3)$ (see \cite[Lemmas 4.3 and 4.4]{velcic}).

Indeed, \eqref{eq:star-star-limsup} follows from \eqref{eq:star-limsup} setting $$V:=(\nabla' u)g+\alpha n_u,$$
and in view of the cancellations due to \eqref{eq:normal1}.
Conversely, \eqref{eq:star-limsup} is obtained from \eqref{eq:star-star-limsup} defining $g:= (\nabla' u)^T V$ and $\alpha:=V\cdot n_u$. 

\end{remark}
A key tool in the proof of the limsup inequality \eqref{eq:intro-limsup} is the following lemma, which has been proved in \cite[Lemma 4.3]{hornung.neukamm.velcic} (see also \cite{hornung2}, \cite{hornung}, \cite{hornung.lewicka.pakzad}, \cite{pakzad}, and \cite{schmidt}). Again, the arguments in the previous sections of this paper continue to hold if $\omega$ is a bounded Lipschitz domain. The piecewise $C^1$-regularity of $\partial\omega$ is necessary for the proof of the limsup inequality \eqref{eq:intro-limsup} (although it can be slightly relaxed as in \cite{hornung}), since it is required in order to obtain the following density result.
\begin{lemma}
\label{lemma:approx-B}
The set $\mathcal{A}(\omega)$ is dense in $W^{2,2}_{R}(\omega;\R^3)$ in the strong $W^{2,2}$ topology.
\end{lemma}
Before we prove the limsup inequality \eqref{eq:intro-limsup}, we state a lemma and a corollary that guarantee the continuity of the relaxations (defined in \eqref{eq:def-q-hom-bar-gamma1-infty}--\eqref{eq:def-q-hom-gamma1-infty}) of the quadratic map ${\mathscr{Q}}$ introduced in (H4). The proof of Lemma \ref{lemma:Q-relaxed-continuous} is a combination of \cite[Proof of Lemma 4.2]{hornung.neukamm.velcic}, \cite[Proof of Lemma 2.10]{neukamm.velcic} and \cite[Lemma 4.2]{velcic}. Corollary \ref{cor:Q-relaxed} is a direct consequence of Lemma  \ref{lemma:Q-relaxed-continuous}.
\begin{lemma}
\label{lemma:Q-relaxed-continuous} Let $\overline{\mathscr{Q}}^{\gamma_1}_{\rm hom}$ and ${\mathscr{Q}}_{\rm hom}$ be the maps defined in \eqref{eq:def-q-hom-bar-0-infty}-\eqref{eq:def-q-hom-gamma1-infty}, and let $\gamma_2=+\infty$.
\begin{enumerate}
\item[(i)]Let $0<\gamma_1<+\infty$. Then for every $A\in\M^{2\times 2}_{\rm sym}$ there exists a unique pair 
$$(B,\phi_1)\in \M^{2\times 2}_{\rm sym}\times  W^{1,2}((-\tfrac12,\tfrac12);W^{1,2}_{\rm per}(Q;\R^3))$$
 with $$\int_{\big(-\tfrac12,\tfrac12\big)\times Q}\phi_1(x_3,y)\,dy\,dx_3=0,$$ such that
\begin{align*}
\overline{\mathscr{Q}}_{\rm hom}^{\gamma_1}(A)&=\int_{\big(-\tfrac12,\tfrac12\big)\times Q}{\mathscr{Q}}_{\rm hom}\Big(y,\Big(\begin{array}{cc}x_3A+B&0\\0&0\end{array}\Big)\\
&\quad+{\rm sym\,}\Big(\nabla_y \phi_1(x_3,y)\Big|\frac{\partial_{x_3}\phi_1(x_3,y)}{\gamma_1}\Big)\Bigg).
\end{align*}
The induced mapping $$A\in\M^{2\times 2}_{\rm sym}\mapsto (B(A),\phi_1(A))\in \M^{2\times 2}_{\rm sym}\times  W^{1,2}((-\tfrac12,\tfrac12);W^{1,2}_{\rm per}(Q;\R^3))$$ is bounded and linear.\\

\item[(ii)]Let $\gamma_1=+\infty$. Then for every $A\in\M^{2\times 2}_{\rm sym}$ there exists a unique triple 
$$(B,d,\phi_1)\in \M^{2\times 2}_{\rm sym}\times L^2((-\tfrac12,\tfrac12);\R^3)\times  L^2((-\tfrac12,\tfrac12);W^{1,2}_{\rm per}(Q;\R^3))$$
 with $$\int_{\left(-\tfrac12,\tfrac12\right)\times Q}\phi_1(x_3,y)\,dy\,dx_3=0,$$ such that
\begin{align*}
\overline{\mathscr{Q}}_{\rm hom}^{\infty}(A)&=\int_{\left(-\tfrac12,\tfrac12\right)\times Q}{\mathscr{Q}}_{\rm hom}\left(y,\left(\begin{array}{cc}x_3A+B&0\\0&0\end{array}\right)\right.\\
&\qquad+{\rm sym\,}(\nabla_y \phi_1(x_3,y)|d(x_3))\Bigg).
\end{align*}
The induced mapping $A\in\M^{2\times 2}_{\rm sym}\mapsto (B(A),d(A),\phi_1(A))\in \M^{2\times 2}_{\rm sym}\times L^2((-\tfrac12,\tfrac12);\R^3)\times  L^{2}((-\tfrac12,\tfrac12);W^{1,2}_{\rm per}(Q;\R^3))$ is bounded and linear.\\

\item[(iii)]Let $\gamma_1=0$. Then for every $A\in\M^{2\times 2}_{\rm sym}$ there exists a unique 6-tuple 
$$(B,\xi, \eta, g_1,g_2,g_3)$$ with $B\in \M^{2\times 2}_{\rm sym}$, $\xi\in L^2\left(\iv; W^{1,2}_{\rm per}(Q;\R^2)\right)$, $\eta\in W^{2,2}_{\rm per}(Q)$, $g_i\in L^2\left(\iv;\times Q\right)$, $i=1,2,3$, such that
\begin{align*}
\overline{\mathscr{Q}}^{0}_{\rm hom}(A)&=\int_{\left(-\tfrac12,\tfrac12\right)\times Q}{\mathscr{Q}}_{\rm hom}\left(y,\left(\begin{array}{cc}x_3A+B&0\\0&0\end{array}\right)\right.\\
&\left. \quad+{\rm sym\,}\left(\begin{array}{cc}{\rm sym }\,\nabla_y\xi(x_3,y)+x_3\nabla_y^2\eta(y)&g_1(x_3,y)\\
\phantom{g}&g_2(x_3,y)\\
g_1(x_3,y)\quad g_2(x_3,y)&g_3(x_3,y)\end{array}\right)\right).
\end{align*}
 The induced mapping $$A\mapsto (B(A),\xi(A),\eta(A),g_1(A), g_2(A),g_3(A))$$ from $\M^{2\times 2}_{\rm sym}$ to $\M^{2\times 2}_{\rm sym}\times L^2((-\tfrac12,\tfrac12);\R^3)\times W^{2,2}_{\rm per}(Q)\times L^{2}((-\tfrac12,\tfrac12)\times Q;\R^3)$ is bounded and linear.
 
For a.e. $y\in Q$ and for every $C\in\mthree_{\rm sym}$ there exists a unique $\phi_2\in W^{1,2}_{\rm per}(Q;\R^3)$, with $\int_Q\phi_2(z)\,dz=0$, such that
$${\mathscr{Q}}_{\rm hom}(y,C)=\int_{Q}{\mathscr{Q}}\left(y,z, C+{\rm sym\,}\big(\nabla \phi_2(z)\big|0\big)\right).$$ 
The induced mapping 
$$C\in\mthree_{\rm sym}\mapsto \phi_2(C)\in W^{1,2}_{\rm per}(Q;\R^3)$$
is bounded and linear. Furthermore, the induced operator
$$P:L^2\left((-\tfrac12,\tfrac12)\times Q;\mthree\right)\to L^2\left((-\tfrac12,\tfrac12)\times Q; W^{1,2}_{\rm per}(Q;\R^3)\right),$$
defined as 
$$P(C):=\phi_2(C)\quad\text{for every }C\in L^2\left((-\tfrac12,\tfrac12)\times Q;\mthree\right)$$ is bounded and linear.
\end{enumerate}
\end{lemma}
\begin{corollary}
\label{cor:Q-relaxed}
Let $\gamma_1\in [0,+\infty]$. The map $\overline{\mathscr{Q}}_{\rm hom}^{\gamma_1}$ is continuous, and there exist $c_1(\gamma_1)\in (0,+\infty)$ such that
$$\frac{1}{c_1}|F|^2\leq \overline{\mathscr{Q}}_{\rm hom}^{\gamma_1}(F)\leq c_1|F|^2$$
for every $F\in\M^{2\times 2}_{\rm sym}$. 
\begin{enumerate}
\item[(i)]If $0<\gamma_1<+\infty$, then for every $A\in L^2(\omega;\M^{2\times 2}_{\rm sym})$ there exists a unique triple $(B,\phi_1,\phi_2)\in L^2\left(\omega;\M^{2\times 2}_{\rm sym}\right)\times L^2\left(\omega;W^{1,2}\left(\left(-\tfrac12,\tfrac12\right);W^{1,2}_{\rm per}(Q;\R^3)\right)\right)\times L^2(\Omega\times Q; W^{1,2}_{\rm per}(Q;\R^3))$ 
such that
\begin{align*}&\frac{1}{12}\int_{\omega}\overline{\mathscr{Q}}_{\rm hom}^{\gamma_1}(A(x'))\,dx'=\int_{\Omega}\overline{\mathscr{Q}}_{\rm hom}^{\gamma_1}(x_3A(x'))\,dx\\
&\quad=\int_{\Omega\times Q}{\mathscr{Q}}_{\rm hom}\left(y,\left(\begin{array}{cc}x_3A(x')+B(x')&0\\0&0\end{array}\right)\right.\\
&\qquad+\left.{\rm sym\,}\left(\nabla_y \phi_1(x,y)\Big|\frac{\partial_{x_3}\phi_1(x,y)}{\gamma_1}\right)\right)\,dy\,dx\\
&\quad=\int_{\Omega\times Q\times Q}{\mathscr{Q}}\left(y,z,\left(\begin{array}{cc}x_3A(x')+B(x')&0\\0&0\end{array}\right)\right.\\
&\qquad+{\rm sym\,}\left(\nabla_y \phi_1(x,y)\Big|\frac{\partial_{x_3}\phi_1(x,y)}{\gamma_1}\right)\\
&\qquad+{\rm sym\,}\left(\left.\nabla_z \phi_2(x,y,z)\right|0\right)\Big)\,dz\,dy\,dx.
\end{align*}
\item[(ii)]If $\gamma_1=+\infty$, then for every $A\in L^2(\omega;\M^{2\times 2}_{\rm sym})$ there exists a unique 4-tuple $(B,d,\phi_1,\phi_2)\in L^2(\omega;\M^{2\times 2}_{\rm sym})\times L^2(\Omega;\R^3)\times L^2(\Omega;W^{1,2}_{\rm per}(Q;\R^3))\times L^2(\Omega\times Q; W^{1,2}_{\rm per}(Q;\R^3))$
such that
\begin{align*}&\frac{1}{12}\int_{\omega}\overline{\mathscr{Q}}_{\rm hom}^{\infty}(A(x'))\,dx'=\int_{\Omega}\overline{\mathscr{Q}}_{\rm hom}^{\infty}(x_3A(x'))\,dx'\\
&\quad=\int_{\Omega\times Q}{\mathscr{Q}}_{\rm hom}\left(y,\left(\begin{array}{cc}x_3A(x')+B(x')&0\\0&0\end{array}\right)
+{\rm sym\,}\left(\nabla_y \phi_1(x,y)|d(x)\right)\right)\,dy\,dx\\
&\quad=\int_{\Omega\times Q\times Q}{\mathscr{Q}}\left(y,z,\left(\begin{array}{cc}x_3A(x')+B(x')&0\\0&0\end{array}\right)
+{\rm sym\,}(\nabla_y \phi_1(x,y)|d(x)\right)\\
&\qquad+{\rm sym\,}\left(\left.\nabla_z \phi_2(x,y,z)\right|0\right)\big)\Big)\,dz\,dy\,dx.
\end{align*}
\item[(iii)]If $\gamma_1=0$, then for every $A\in L^2(\omega;\M^{2\times 2}_{\rm sym})$ there exists a unique 7-tuple\\
 $(B,\xi,\eta,g_1,g_2,g_3,\phi)\in L^2(\omega;\M^{2\times 2}_{\rm sym})\times L^2(\Omega;W^{1,2}_{\rm per}(Q;\R^2))\times L^2(\Omega;W^{2,2}_{\rm per}(Q))\times L^2(\Omega\times Q;\R^3)\times L^2(\Omega\times Q; W^{1,2}_{\rm per}(Q;\R^3))$,
such that
\begin{align*}&\frac{1}{12}\int_{\omega}\overline{\mathscr{Q}}_{\rm hom}^{0}(A(x'))\,dx'=\int_{\Omega}\overline{\mathscr{Q}}_{\rm hom}^{0}(x_3A(x'))\,dx'\\
&\quad=\int_{\Omega \times Q}{\mathscr{Q}}_{\rm hom}\left(y,\left(\begin{array}{cc}x_3A(x')+B(x')&0\\0&0\end{array}\right)\right.\\
&\qquad+{\rm sym\,}\left.\left(\begin{array}{cc}{\rm sym }\,\nabla_y\xi(x,y)+x_3\nabla_y^2\eta(x',y)&g_1(x,y)\\
\phantom{g}&g_2(x,y)\\
g_1(x,y)\quad g_2(x,y)&g_3(x,y)\end{array}\right)\right)\\
&\quad=\int_{\Omega\times Q\times Q}{\mathscr{Q}}\left(y,z,\left(\begin{array}{cc}x_3A(x')+B(x')&0\\0&0\end{array}\right)\right.\\
&\qquad+{\rm sym\,}\left(\begin{array}{cc}{\rm sym }\,\nabla_y\xi(x,y)+x_3\nabla_y^2\eta(x',y)&g_1(x,y)\\
\phantom{g}&g_2(x,y)\\
g_1(x,y)\quad g_2(x,y)&g_3(x,y)\end{array}\right)\\
&\qquad+{\rm sym\,}\left(\nabla_z \phi_2(x,y,z)\big|0\right)\Bigg)\,dz\,dy\,dx.
\end{align*}
\end{enumerate}
\end{corollary}

We now prove that the lower bound obtained in Section \ref{section:liminf} is optimal.
\begin{theorem}
\label{thm:limsup}
Let $\gamma_1\in [0,+\infty]$. Let $\overline{\mathscr{Q}}_{\rm hom}^{\gamma_1}$ and ${\mathscr{Q}_{\rm hom}}$ be the maps defined in \eqref{eq:def-q-hom-bar-0-infty}--\eqref{eq:def-q-hom-gamma1-infty}, let \mbox{$u\in W^{2,2}_{R}(\omega;\R^3)$} and let $\Pi^u$ be the map introduced in \eqref{eq:def-Pi}. Then there exists a sequence $\{u^h\}\subset W^{1,2}(\Omega;\R^3)$ such that
\be{eq:limsup-ineq}
\limsup_{h\to 0}\frac{\mathcal{E}^h(u^h)}{h^2}\leq \frac{1}{12}\int_{\omega}\overline{\mathscr{Q}}_{\rm hom}^{\gamma_1}(\Pi^u(x'))\,dx'.
\ee
\end{theorem}
\begin{proof} The proof is an adaptation of \cite[Proof of Theorem 2.4]{hornung.neukamm.velcic} and \cite[Proof of Theorem 2.4]{velcic}. We outline the main steps in the cases $0<\gamma_1<+\infty$ and $\gamma_1=0$ for the convenience of the reader. The proof in the case $\gamma_1=+\infty$ is analogous.\\

\noindent\textbf{Case 1}: $0<\gamma_1<+\infty$ and $\gamma_2=+\infty$.\\
By Lemma \ref{lemma:approx-B} and Corollary \ref{cor:Q-relaxed} it is enough to prove the theorem for $u\in\mathcal{A}(\omega)$. By Corollary \ref{cor:Q-relaxed} there exist $B\in L^2(\omega;\M^{2\times 2})$, $\phi_1\in L^{2}(\omega; W^{1,2}((-\tfrac12,\tfrac12);W^{1,2}_{\rm per}(Q;\R^3))$, and $\phi_2\in L^{2}(\Omega\times Q;W^{1,2}_{\rm per}(Q;\R^3))$ such that
\begin{align*}
 &\tfrac{1}{12}\int_{\omega}\overline{\mathscr{Q}}_{\rm hom}^{\gamma_1}(\Pi^u(x'))\,dx'\\
 &\quad=\intoyz{\mathscr{Q}\left(y,z,\left(\begin{array}{cc}{\rm sym\,}B(x')+x_3\Pi^u(x')&0\\0&0\end{array}\right)\right.\\
&\qquad\left.+{\rm sym\,}\left(\nabla_y \phi_1(x,y)\left|\frac{\partial_{x_3}\phi_1(x,y)}{\gamma_1}\right.\right)+{\rm sym\,}\big(\nabla_z \phi_2(x,y,z)\big|0\big)\right)}.
\end{align*}
Since $B$ depends linearly on $\Pi^u$ by Lemma \ref{lemma:Q-relaxed-continuous}, in particular there holds
$$\{x':\Pi^u(x')=0\}\subset \{x':\,B(x')=0\}.$$
By Lemma \ref{lemma:Q-relaxed-continuous}, we can argue by density and we can assume that $B\in C^{\infty}(\bar{\omega};\M^{2\times 2})$, $B=0$ in a neighborhood of $\{x':\Pi^u(x')=0\}$, $\phi_1\in C^{\infty}_c(\omega; C^{\infty}((-\tfrac12,\tfrac12); C^{\infty}(Q;\R^3))$, and $\phi_2\in C^{\infty}_c(\omega\times Q; C^{\infty}(Q;\R^3))$. In addition, since $u\in\mathcal{A}(\omega)$, by \eqref{eq:star-limsup} there exist $\alpha\in C^{\infty}(\bar{\omega})$, and $g\in C^{\infty}(\bar{\omega};\R^2)$ such that
$$B={\rm sym}\,\nabla' g+\alpha\Pi^u.$$

Set
\bmm
&v^h(x):=u(x')+h((x_3+\alpha(x'))n_u(x')+(g(x')\cdot\nabla')y(x')),\\
&R(x'):=(\nabla'u(x')|n_u(x')),\\
&b(x'):=-\left(\begin{array}{c}\partial_{x_1}\alpha(x')\\\partial_{x_2}\alpha(x')\end{array}\right)+\Pi^u(x')g(x'),
\end{align*}
and let
$$u^h(x):=v^h(x')+h\ep(h)\tilde{\phi}_1\Big(x,\frac{x'}{\ep(h)}\Big)+h\ep^2(h)\tilde{\phi}_2\Big(x,\frac{x'}{\ep(h)},\frac{x'}{\ep^2(h)}\Big)$$
for a.e. $x\in\Omega$, where
$$\tilde{\phi}_1:=R\Big(\phi_1+\gamma_1x_3\Big(\begin{array}{c}b\\0\end{array}\Big)\Big)\quad\text{and}\quad\tilde{\phi}_2:=R\phi_2.$$
Arguing similarly to \cite[Proof of Theorem 2.4 (upper bound)]{hornung.neukamm.velcic}, it can be shown that \eqref{eq:limsup-ineq} holds.\\

\noindent\textbf{Case 2}: $\gamma_1=0$ \textbf{and }$\gamma_2=+\infty$.\\
By Lemma \ref{lemma:approx-B} and Corollary \ref{cor:Q-relaxed} it is enough to prove the theorem for $u\in\mathcal{A}(\omega)$. By Corollary \ref{cor:Q-relaxed} there exist $B\in L^2(\omega;\M^{2\times 2}_{\rm sym})$, $\xi\in L^2(\Omega;W^{1,2}_{\rm per}(Q;\R^2))$, $\eta\in L^2(\Omega;W^{2,2}_{\rm per}(Q))$, $g_i\in L^2(\Omega\times Y)$, $i=1,2,3$, and $\phi\in L^2(\Omega\times Q; W^{1,2}_{\rm per}(Q;\R^3))$ such that
\begin{align*}
&\frac{1}{12}\int_{\omega}\overline{\mathscr{Q}}_{\rm hom}^{0}(\Pi^u(x'))\,dx'\\
&\quad=\int_{\Omega\times Q\times Q}{\mathscr{Q}}\left(y,z,\left(\begin{array}{cc}x_3\Pi^u(x')+B(x')&0\\0&0\end{array}\right)\right.\\
&\qquad+{\rm sym\,}\left(\begin{array}{cc}{\rm sym }\,\nabla_y\xi(x,y)+x_3\nabla_y^2\eta(x',y)&g_1(x,y)\\
\phantom{g}&g_2(x,y)\\
g_1(x,y)\quad g_2(x,y)&g_3(x,y)\end{array}\right)\\
&\qquad+{\rm sym\,}\left(\left.\nabla_z \phi_2(x,y,z)\right|0\right)\Bigg)\,dz\,dy\,dx.
\end{align*}
By the linear dependence of $B$ on $\Pi^u$, in particular there holds
$$\{x':\Pi^u(x')=0\}\subset \{x':\,B(x')=0\}.$$
By density, we can assume that $B\in C^{\infty}(\bar{\omega};\M^{2\times 2})$, $\xi\in C^{\infty}_c(\omega; C^{\infty}_{\rm per}(Q;\R^2))$, $\eta\in C^{\infty}_c(\omega; C^{\infty}_{\rm per}(Q))$, and $g_i\in C^{\infty}_c(\omega; C^{\infty}_{\rm per}\big(\iv\times Q\big)),$ $i=1,2,3$. Since $u\in\mathcal{A}(\omega)$, by \eqref{eq:star-star-limsup} there exists a displacement $V\in C^{\infty}(\bar{\omega};\R^2)$ such that
$$B={\rm sym}\,((\nabla' u)^T\nabla'V).$$
Set
\begin{align*}
&v^h(x):=u(x')+hx_3n_u(x')=h(V(x')+hx_3\mu(x')),\\
&\mu(x'):=(Id-n_u(x')\otimes n_u(x'))(\partial_{1}V(x')\wedge\partial_2 u(x')+\partial_1 u(x')\wedge\partial_2 V(x')),\\
&R(x'):=(\nabla' u(x')|n_u(x')),
\end{align*}
and let
\begin{align*}
u^h(x)&:=v^h(x)-\ep^2(h)n_u(x')\eta\Big(x',\frac{x'}{\ep(h)}\Big)\\
&\quad+h\ep^2(h)x_3 R(x')\left(
\begin{array}{c}\partial_{x_1}\eta\left(x',\frac{x'}{\ep(h)}\right)+\frac{1}{\ep(h)}\partial_{y_1}\eta\left(x',\frac{x'}{\ep(h)}\right)\\
\partial_{x_2}\eta\left(x',\frac{x'}{\ep(h)}\right)+\frac{1}{\ep(h)}\partial_{y_2}\eta\left(x',\frac{x'}{\ep(h)}\right)\end{array}\right)\\
&\quad+h\ep(h)R(x')\left(\begin{array}{c}\xi\Big(x',\frac{x'}{\ep(h)}\Big)\\0\end{array}\right)\\
&\quad+h^2\int_{-\tfrac12}^{x_3}R(x')g\Big(x',t,\frac{x'}{\ep(h)}\Big)\,dt+h\ep^2(h)R(x')\phi\Big(x,\frac{x'}{\ep(h)},\frac{x'}{\ep^2(h)}\Big),
\end{align*}
for a.e. $x\in\Omega$. The proof of \eqref{eq:limsup-ineq} is a straightforward adaptation of \cite[Proof of Theorem 2.4 (Upper bound)]{velcic}.
\end{proof}
\begin{proof}[Proof of Theorem \ref{thm:main-result}]
Theorem \ref{thm:main-result} follows now by Theorem \ref{thm:liminf} and Theorem \ref{thm:limsup}.
\end{proof}

\section{Appendix}
In this section we collect a few results which played an important role in the proof of Theorem \ref{thm:main-result}. We recall that in Case 2, we claimed that the maps $R^h$ are piecewise constant on cubes of the form $Q(\ep(h)z,\ep(h))$, $z\in\Z^2$. Indeed, this holds if we show that for every $z\in\Z^2$ there exists $z'\in\Z^2$ such that
$$Q(\ep(h)z,\ep(h))\subset Q(\delta(h)z',\delta(h))$$
or, equivalently, with $m:=\frac{\delta(h)}{\ep(h)}\in\N$,
\be{eq:right-inclusion}\Big(z-\frac12,z+\frac12\Big)\subset m\Big(z'-\frac12,z'+\frac12\Big).\ee
The next lemma attests that this holds provided $m$ is odd.
\begin{lemma} 
\label{lemma:integers-appendix}
Let $a\in\N_0$. Then for every $z\in\Z$ there exists $z'\in\Z$ such that
\eqref{eq:right-inclusion} holds with $m=2a+1$.
\end{lemma}
\begin{proof}
Without loss of generality we may assume that $z\in\N_0$ (the case in which $z<0$ is analogous). Solving \eqref{eq:right-inclusion} is equivalent to finding $z'\in\Z$ such that
\be{eq:appendix-star}
\begin{cases}
z-\frac12\geq (2a+1)z'-\frac{(2a+1)}{2},&\\
z+\frac12\leq (2a+1)z'+\frac{(2a+1)}{2},&
\end{cases}
\ee
that is
\be{eq:appendix-star-star}
\begin{cases}
z\geq (2a+1)z'-a,&\\
z\leq (2a+1)z'+a.&
\end{cases}
\ee
Let $n,l\in\N_0$ be such that $z=n(2a+1)+l$ and
\be{eq:appendix-point}
l<2a+1.
\ee
Then \eqref{eq:appendix-star-star} is equivalent to
\be{eq:appendix-last}
\begin{cases}
n(2a+1)+l+a\geq (2a+1)z',&\\
n(2a+1)+l-a\leq (2a+1)z'.&
\end{cases}
\ee
Now, if $0\leq l\leq a$ it is enough to choose $z'=n$. If $l>a$, the result follows setting $z':=n+1$. Indeed, with $a+1>r>1\in\N$ such that $l=a+r$, \eqref{eq:appendix-last} simplifies as
$$\begin{cases}
n(2a+1)+2a+r\geq (2a+1)(n+1),&\\
n(2a+1)+r\leq (2a+1)(n+1),&
\end{cases}$$
that is
$$\begin{cases}
2a+r\geq 2a+1&\\
r\leq 2a+1,
\end{cases}$$
which is trivially satisfied.
\end{proof}
\begin{remark}
\label{rk:case3}
By Lemma \ref{lemma:integers-appendix} it follows that, setting $p:=\frac{\delta(h)}{\ep^2(h)}$ and provided $p$ is odd, for every $z\in\Z^2$ there exists $z'\in\Z^2$ such that
$$Q(\ep^2(h)z,\ep^2(h))\subset Q(\delta(h)z,\delta(h)).$$
This observation allowed us to construct the sequence $\{R^h\}$ in Case 3 of the proof of Theorem \ref{thm:main-result}.
\end{remark}
\begin{remark}
We point out that if $m$ is even there may be $z\in\Z$ such that \eqref{eq:right-inclusion} fails to be true for every $z'\in\Z$,  i.e.
$$\Big(z-\frac12,z+\frac12\Big)\not\subseteq \Big(mz'-\frac{m}{2},mz'+\frac{m}{2}\Big).$$
Indeed, if $m$ is even, then $z=\tfrac32 m\in\N$ and \eqref{eq:appendix-star} becomes
$$\begin{cases}
\frac32 m-\frac12\geq mz'-\frac{m}{2}&\\
\frac32 m+\frac12\leq mz'+\frac{m}{2},&
\end{cases}$$
which in turn is equivalent to
$$z'\in\Big[1+\frac{1}{2m},2-\frac{1}{2m}\Big].$$
This last condition leads to a contradiction as
$$\Big[1+\frac{1}{2m},2-\frac{1}{2m}\Big]\cap\Z=\emptyset\quad\text{for every } m\in\N.$$
\end{remark}
We conclude the Appendix with a result that played a key role in the identification of the limit elastic stress, and in the proof of the liminf and limsup inequalities \eqref{eq:intro-liminf} and \eqref{eq:intro-limsup}. We omit its proof, as it follows by \cite[Lemma 4.3]{neukamm.velcic}.
\begin{lemma}
\label{lemma:mod3scales-lsc}
Let ${\mathscr{Q}}:\R^2\times\R^2\times \mthree\to[0,+\infty)$ be such that
\begin{enumerate}
\item[(i)] ${\mathscr{Q}}(y,z,\cdot)$ is continuous for a.e. $y,z\in \R^2$,
\item[(ii)] ${\mathscr{Q}}(\cdot,\cdot,F)$ is $Q\times Q$-periodic and measurable for every $F\in\mthree$,
\item[(iii)] for a.e. $y,z\in\R^2$, the map ${\mathscr{Q}}(y,z,\cdot)$ is quadratic on $\mthree_{\rm sym}$, and satisfies
$$\frac{1}{C}|{\rm sym F}|^2\leq {\mathscr{Q}}(y,z,F)={\mathscr{Q}}(y,z,{\rm symF})\leq C|{\rm sym F}|^2$$
for all $F\in\mthree$, and some $C>0$.
\end{enumerate}

 Let $\{E^h\}\subset L^2(\Omega;\mthree)$ and $E\in L^2(\Omega\times Q\times Q;\mthree)$ be such that
$$E^h\wks E\quad\text{weakly dr-3-scale}.$$
Then
$$\liminf_{h\to 0}\int_{\Omega}{\mathscr{Q}}\Big(\frac{x'}{\ep(h)},\frac{x'}{\ep^2(h)}, E^h(x)\Big)\,dx\geq \intoyz{\mathscr{Q}(y,z,E(x,y,z))}.$$
If in addition 
$$E^h\stts E\quad\text{strongly dr-3-scale},$$
then
$$\lim_{h\to 0}\int_{\Omega}{\mathscr{Q}}\Big(\frac{x'}{\ep(h)},\frac{x'}{\ep^2(h)}, E^h(x)\Big)\,dx= \intoyz{\mathscr{Q}(y,z,E(x,y,z))}.$$
\end{lemma}

\section{Concluding Remarks}
The rigorous identification of two-dimensional models for thin three-dimensional structures is a classical question in mechanics of materials. Recently, in \cite{hornung.neukamm.velcic}, \cite{neukamm.velcic} and \cite{velcic}, simultaneous homogenization and dimension reduction for thin plates has been studied, under physical growth conditions for the energy density, and in the situation in which one periodic in-plane homogeneity scale arises.

In this paper we deduced a multiscale version of \cite{hornung.neukamm.velcic} and \cite{velcic}, extending the analysis to the case in which two periodic in-plane homogeneity scale are present, in the framework of Kirchhoff's nonlinear plate theory. Denoting by $h$ the thickness of the plate, and by $\ep(h)$ and $\ep^2(h)$ the two periodicity scales, we provided a characterization of the effective energy in the regimes
$$\lim_{h\to 0}\frac{h}{\ep(h)}:=\gamma_1\in[0,+\infty]\quad\text{and}\quad\lim_{h\to 0}\frac{h}{\ep^2(h)};=\gamma_2=+\infty.$$
The analysis relies on multiscale convergence methods and on a careful study of the multiscale limit of the sequence of linearized three-dimensional stresses, based on Friesecke, James and M\"uller's rigidity estimate (\cite[Theorem 4.1]{friesecke.james.muller}).

The identification of the reduced models for $\gamma_1=0$ and $\gamma_2\in[0,+\infty)$ remains an open problem.
\section*{acknowledgements}
The authors thank the Center for Nonlinear Analysis (NSF Grant No. DMS-0635983), where this research was carried out. The research of L. Bufford,
E. Davoli, and I. Fonseca was funded by the National Science Foundation under Grant No. DMS-
0905778. L. Bufford and I. Fonseca were also supported by the National Science Foundation under Grant No. DMS-1411646. E. Davoli and I. Fonseca
 acknowledge support of the National Science Foundation under the PIRE Grant No.
OISE-0967140.

\end{document}